\documentclass[a4paper,11pt,reqno]{amsart}

\usepackage[headings]{fullpage}

\usepackage{mathtools}
\usepackage{amssymb}
\usepackage{tensor}
\usepackage{tikz}
\usetikzlibrary{cd}

\usepackage[backref]{hyperref}

\usepackage[nobysame,alphabetic,initials,msc-links]{amsrefs}

\DefineSimpleKey{bib}{how}

\renewcommand{\eprint}[1]{#1}
\BibSpec{misc}{%
  +{}{\PrintAuthors}  {author}
  +{,}{ \textit}      {title}
  +{,}{ }             {how}
  +{}{ \parenthesize} {date}
  +{,} { available at \eprint}        {eprint}
  +{,}{ available at \url}{url}
  +{,}{ }             {note}
  +{.}{}              {transition}
}

\mathchardef\mhyph="2D
\DeclareSymbolFont{stmry}{U}{stmry}{m}{n}
\DeclareMathSymbol\stmryolt\mathbin{stmry}{"3C}
\DeclareMathSymbol\stmryogt\mathbin{stmry}{"3D}

\numberwithin{equation}{section}

\newtheorem{theorem}{Theorem}[section]
\newtheorem{corollary}[theorem]{Corollary}
\newtheorem{lemma}[theorem]{Lemma}
\newtheorem{proposition}[theorem]{Proposition}
\newtheorem{theoremAlph}{Theorem}

\theoremstyle{remark}
\newtheorem{remark}[theorem]{Remark}
\newtheorem{example}[theorem]{Example}

\theoremstyle{definition}
\newtheorem{definition}[theorem]{Definition}

\newcommand{\C}{\mathbb{C}}

\newcommand{\cA}{\mathcal{A}}
\newcommand{\cB}{\mathcal{B}}
\newcommand{\cC}{\mathcal{C}}
\newcommand{\cD}{\mathcal{D}}

\newcommand{\cK}{\mathcal{K}}
\newcommand{\cI}{\mathcal{I}}
\newcommand{\cL}{\mathcal{L}}
\newcommand{\cM}{\mathcal{M}}
\newcommand{\cN}{\mathcal{N}}
\newcommand{\cO}{\mathcal{O}}

\newcommand{\cR}{\mathcal{R}}

\DeclareMathOperator{\Ad}{Ad}
\DeclareMathOperator{\Alg}{Alg}
\DeclareMathOperator{\End}{End}

\DeclareMathOperator{\ind}{ind}
\DeclareMathOperator{\Irr}{Irr}

\DeclareMathOperator{\cMod}{Mod^c}
\DeclareMathOperator{\Nat}{Nat}

\DeclareMathOperator{\Rep}{Rep}

\DeclareMathOperator{\Vect}{Vec}

\DeclareMathOperator{\tr}{tr}
\DeclareMathOperator{\Id}{Id}
\DeclareMathOperator{\id}{id}
\DeclareMathOperator{\Img}{Im}

\newcommand{\norm}[1]{\left \| #1 \right \|}
\newcommand{\linnprod}[3]{\tensor[_{#1}]{\langle #2, #3 \rangle}{}}
\newcommand{\vnotimes}{\mathbin{\bar{\otimes}}}

\newcommand\op{\mathrm{op}}
\newcommand{\Hilb}{\mathrm{Hilb}}
\newcommand{\FH}{\mathrm{FH}}
\newcommand{\CB}{\mathcal{CB}}
\newcommand{\cCB}{\mathcal{CB}^c}
\newcommand{\CBOS}{\mathcal{CBOS}}

\title[Injectivity for module categories]{Injectivity for algebras and categories with quantum symmetry}

\date{April 2025, restructuring; July 21, 2023, minor revision; May 13, 2022}

\author{Lucas Hataishi}
\address{Universitetet i Oslo}
\curraddr{University of Oxford}
\email{lucas.yudihataishi@maths.ox.ac.uk}

\author{Makoto Yamashita}
\address{Universitetet i Oslo}
\email{makotoy@math.uio.no}

\thanks{Supported by the NFR project 300837 ``Quantum Symmetry''.}

\begin{document}

\begin{abstract}
We establish the existence of injective envelopes for unital Yetter--Drinfeld C$^*$-algebras, and a related class of bimodule categories over rigid C$^*$-tensor categories.
This implies monoidal invariance for boundary actions of Drinfeld doubles of compact quantum groups.
\end{abstract}

\maketitle

\section{Introduction}

In this paper, we study completely positive maps and injectivity for Yetter--Drinfeld algebras over compact quantum groups, and module categories over rigid C$^*$-tensor categories.

The study of injectivity of operator systems and operator spaces led to surprisingly rich applications to the structure theory of operator algebras old and new.
While the foundation of the theory goes back to the 1950s, one early breakthrough in this context is Arveson's result on the injectivity of $\cB(H)$~\cite{MR253059}.
This, together with subsequent works by Choi and Effros~\citelist{\cite{MR0376726}\cite{ce1}} among others, proved the field to be a fruitful framework at the intersection of abstract functional analysis and more ``applied'' fields such as quantum information.

Following this direction, Hamana took on systematic study of injectivity~\citelist{\cite{h4}\cite{h1}\cite{h2}\cite{h3}\cite{MR2985658}}, defining and proving the existence and uniqueness of injective envelopes of operator systems, $C^*$-algebras, and $C^*$-dynamical systems with dynamics given by a discrete group, and more generally by a Hopf--von Neumann algebra.
Roughly speaking, his construction builds on two parts: the first is to show that the given object embeds into a bigger injective object, and the next is to show that a `smallest' injective one containing the original object is obtained as the image of minimal idempotent in a convex semigroup of completely positive maps acting on the injective one.

In another direction, Furstenberg and his school studied the probabilistic notion of boundary actions of semisimple Lie groups~\citelist{\cite{furstenberg1}\cite{moore}}, and more generally of locally compact groups~\cite{MR365537}.
A central object in their theory is the notion of  a universal (initial) boundary action of $G$, now known as the Furstenberg boundary $\partial_F(G)$.
In the case of a semisimple Lie group $G$, an important application of boundary theory is to the symmetric space $G/K$ corresponding to $G$, which appears through measure theoretic considerations on $\partial_F(G)$.
Furstenberg also considered another kind of boundary, the Poisson boundary $B(G)$ of $G$, which shows up through the study of harmonic functions, and studied how $\partial_F(G)$ relates to $B(G)$.

While these two theories look completely separate at the outset, it turns out that the $C^*$-algebra of continuous functions on the Furstenberg boundary of a discrete group $G$ agrees with the $G$-injective envelope of the trivial $G$-$C^*$-algebra $\C$~\citelist{\cite{h4}\cite{kk}}.
In recent years, this connection led to striking implications on the structure of group $C^*$-algebras: this includes the equivalence between amenability of a discrete group and triviality of its Furstenberg--Hamana boundary, results on rigidity and $C^*$-simplicity~\citelist{\cite{MR3735864}\cite{kennedy}}.

Probabilistic boundary theory has been brought to the framework of operator algebraic quantum groups in various flavors in the last 30 years~\citelist{\cite{biane}\cite{izumi1}\cite{izumi2}\cite{int}\cite{neshveyevtuset1}\cite{vaesvergnioux}\cite{NY1}}.
More recently, the analogue of Furstenberg--Hamana boundaries in this framework is also achieved by the work of Kalantar, Kasprzak, Skalski, and Vergnioux~\cite{kksv1}.
They define and prove both the existence and the uniqueness of Furstenberg--Hamana boundaries for discrete quantum groups, and relate them to several $C^*$-algebraic concepts such as simplicity, exactness, and existence of KMS-states.

This was further brought to the case of Drinfeld doubles $D(G)$ of compact quantum groups $G$ in a joint work of the first named author with Habbestad and Neshveyev~\cite{HHN1}.
When the compact quantum group is the $q$-deformation of a compact simple group, this Drinfeld double construction can be seen as a quantization of the corresponding complex simple Lie group, hence coming back to a setting close to Furstenberg's original work.
This also provides a conceptual explanation for the equality between the Furstenberg--Hamana boundary and the Poisson boundary of such quantum groups.

In~\cite{HHN1}, besides a detailed account of boundary theory for Yetter--Drinfeld $G$-C$^*$-algebras, a monoidal invariant approach was also initiated, attempting to derive boundary actions from the representation theory of compact quantum groups.
Due to the lack of a categorical description of general Yetter--Drinfeld C$^*$-algebras, the monoidal invariant approach was left incomplete, encompassing only the \emph{braided-commutative} ones.
Later in this paper, we will adapt our techniques to the setting of~\cite{YDpaper} to complete that picture.
We thus prove that if two compact quantum groups $G$ and $G'$ are monoidally equivalent, i.e., if $\Rep (G) \simeq \Rep(G')$, then the cateogries of boundary actions of the respective Drinfeld doubles $D(G)$ and $D(G')$ are equivalent.

Let us explain the structure of our paper.
In Section~\ref{sec:prelim} we fix our conventions and recall some basic results.
Then in Section~\ref{sec:inj-env-YD-G-algs}, we quickly prove our first main result, as follows.

\begin{theoremAlph}[Theorem~\ref{thm:inj-env-YD-G-alg}]\label{thmA}
For a compact quantum group $G$, every continuous unital Yetter--Drinfeld $G$-C$^*$-algebra $A$ admits an injective envelope.
\end{theoremAlph}

This result, showing that the injective envelopes of continuous $D(G)$-actions are still continuous, improves a result of Hamana.

In Section~\ref{sec:inj-mod-cat}, we look at the categorical dual of quantum group actions.
We begin with the study of injective envelope of a pointed module category, i.e., a module category with a fixed object.
Here we make extensive use of the correspondence between algebra objects in a tensor category and cyclic pointed module categories, and the concept of multipliers between pointed module categories introduced in~\cite{jp1}.
This allows us to bring ideas about completely positive maps between C$^*$-algebras to the categorical setting.
In fact, special cases of such maps have aleady been considered in~\citelist{\cite{PV1}\cite{jonesghosh1}\cite{NY2}} to study representation theory and approximation properties, like Haagerup property and property (T), for rigid $C^*$-tensor categories, subfactors and $\lambda$-lattices.

Given a Hilbert space object $H \in \Hilb(\cC)$ (equivalently, an object in the ind-completion of $\cC$ in the sense of~\cite{NY2}), there is a C$^*$-algebra object $\cB(H)$ in $\Vect(\cC)$ playing the role of algebra of bounded operators in $H$.
Our first main result in this setting is the following analogue of Arveson's theorem.

\begin{theoremAlph}[Theorem~\ref{thm:arveson-type-thm-for-mod-cat}]
For any $H \in \Hilb(\cC)$, the C$^*$-$\cC$-module category $\cM_{\cB(H)}$ corresponding to $\cB(H)$ is injective.
\end{theoremAlph}

By a standard argument, this allows us to obtain an analogue of Hamana's theorem for module categories (Theorem~\ref{thm:injective-envelopes-mod-cat}).
We also show that the injectivity of a pointed module category $(\cM, m)$ is actually a property of the module subcategory of $\cM$ generated by $m$ (Proposition~\ref{C-U-injective}) through a generalization of the Choi matrix construction.
In particular, the injectivity of $G$-C$^*$-algebra is invariant under Morita equivalence with respect to finitely generated projective Hilbert modules.

Section~\ref{sec:op-sys-half-br-po-bim} starts by recalling the categorical duals of Yetter--Drinfeld C$^*$-algebras, following~\cite{YDpaper}.
They will be called {\em centrally cyclic bimodule categories}.
We then investigate injectivity and boundary theory of centrally pointed bimodule categories based on the schemes of previous sections.
Here our main result is the following analogue of Theorem~\ref{thmA}.

\begin{theoremAlph}[Theorem~\ref{thm:injective-envelope-cb-mod-cat}]
Every centrally pointed bimodule category has an injective envelope.
\end{theoremAlph}

In the boundary theory for $D(G)$ developed in~\cite{HHN1}, one sees that every boundary action embeds into the Furstenberg--Hamana boundary.
Following a similar idea, we obtain the following generalization in the categorical framework, where a \emph{boundary} of $\cC$ is a centrally pointed module category over $\cC$ such that any ucp multiplier to another category is completely isometric.

\begin{theoremAlph}[Theorem~\ref{thm:charac-bnd}]
Every boundary category $\cM$ of $\cC$ is a centrally pointed subcategory of $\partial_\FH(\cC)$.
\end{theoremAlph}

In Appendix~\ref{app:intr-char-C-star-alg} we give a supplementary result on C$^*$-algebra objects introduced in~\cite{jp1}, establishing an intrinsic characterization of such objects that does not directly refer to a C$^*$-category structure on the associated module category.

\bigskip
\paragraph{Acknowledgements}
We would like to thank the anonymous reviewers for their helpful comments which helped us improve the presentation of the paper.

\section{Preliminaries}\label{sec:prelim}

When $H$ is a Hilbert space, we denote the algebra of bounded operators on $H$ by $\cB(H)$, while the algebra of compact operators is denoted by $\cK(H)$.
The dual Hilbert space is denoted by $\bar{H}$, and we write $\xi \mapsto \bar{\xi}$ for the anti-linear isomorphism $H \to \bar{H}$.

The multiplier algebra of a C$^*$-algebra $A$ is denoted by $\cM(A)$.
We freely identify $\cM(\cK(H))$ with $\cB(H)$.

As tensor product of C$^*$-algebras we always take the minimal tensor product, that we denote by $A \otimes B$.
For von Neumann algebras, their von Neumann algebraic tensor product is denoted by $M \vnotimes N$.

\subsection{Operator systems}

An \emph{operator system} is a unital and selfadjoint closed subspace of a unital C$^*$-algebra, which we write as $1 \in S \subset A$.
Such objects form a category with unital completely positive maps as morphisms.
An operator system is \emph{injective} if it is injective in this category.

When $S$ is an operator system, its \emph{injective envelope} is given by an injective operator system $\cI(S)$ and a complete isometry $\phi \colon S \to \cI(S)$ such that $\phi$ is \emph{essential} in the sense that for any ucp map $\psi \colon \cI(S) \to S'$ to another operator system, $\psi \phi$ is completely isometric if and only if $\psi$ is.

The injective envelope of an operator systems always exists and is unique up to complete order isomorphisms~\cite{h1}, with various generalizations imposing additional structures on operator systems.
A key technical step in the construction of injective envelopes is the following proposition.

\begin{proposition}[\cite{HHN1}*{Proposition 2.1}]\label{minimalidemp}
Assume $X$ is a subspace of a dual Banach space $Y^*$, and $S$ is a convex semigroup of contractive linear maps $X \to X$ such that, if we consider $S$ as a set of maps $X \to Y^*$, then $S$ is closed in the topology of pointwise weak* convergence.
Then there is an idempotent $\phi_0 \in S$ such that
\[
  \phi_0 \psi \phi_0 = \phi_0 \quad (\psi \in S).
\]
Moreover, $\phi_0$ is minimal with respect to the preorder relation $\preceq$ on $S$ defined as
\[
  \phi \preceq \psi \Leftrightarrow \forall x \in X \colon \norm{\phi(x)} \leq \norm{\psi(x)}.
\]
\end{proposition}

\subsection{\texorpdfstring{C$^*$}{C*}-tensor categories}

Here we mostly follow~\citelist{\cite{MR3204665}\cite{NY3}}.

Given a \emph{C$^*$-category} $\cC$, we denote the spaces of morphisms from an object $X$ to another $Y$ as $\cC(X, Y)$.
The involution $\cC(X, Y) \to \cC(Y, X)$ is denoted by $T \mapsto T^*$, and the norm is by $\norm{T}$, so that we have the C$^*$-identity $\norm{T^* T} = \norm{T}^2$.
We also assume that $X \oplus Y$ make sense with structure morphisms given by isometries $X \to X \oplus Y \leftarrow Y$.
Under these assumptions, $\cC(X, Y)$ is naturally a right Hilbert module over the unital C$^*$-algebra $\cC(X) = \cC(X, X)$.
We tacitly assume that $\cC$ is closed under taking subobjects, i.e., any projection in $\cC(X)$ corresponds to a direct summand of $X$.

A \emph{C$^*$-tensor category} is a C$^*$-category endowed with monoidal structure given by: a $*$-bifunctor $\otimes\colon \cC \times \cC \to \cC$, a unit object $1_\cC$, and unitary natural isomorphisms
\begin{align*}
  1_\cC & \otimes U \to U \leftarrow U \otimes 1_\cC, & \Phi\colon (U \otimes V) \otimes W & \to U \otimes (V \otimes W)
\end{align*}
for $U, V, W \in \cC$, satisfying a standard set of axioms.
Without losing generality, we may and do assume that $\cC$ is \emph{strict} so that the above morphisms are identity, and that $1_\cC$ is simple, unless explicitly stated otherwise.

A \emph{rigid} C$^*$-tensor category is a C$^*$-tensor category where any object $U$ has a dual given by: an object $\bar U$ and morphisms
\begin{align*}
  R \colon 1_\cC & \to \bar U \otimes U, & \bar R \colon 1_\cC & \to U \otimes \bar U
\end{align*}
satisfying the conjugate equations for $U$.

When $\cC$ is a C$^*$-tensor category, a \emph{right (C$^*$-)$\cC$-module category} is given by a C$^*$-category $\cM$, together with a $*$-bifunctor $\stmryolt \colon \cM \times \cC \to \cM$ and unitary natural isomorphisms
\begin{align*}
  X & \stmryolt 1_\cC  \to X, & \Psi\colon (X \stmryolt U) \stmryolt V & \to X \stmryolt (U \otimes V)
\end{align*}
for $X \in \cM$ and $U, V \in \cC$, satisfying standard set of axioms.
Again we may and do assume that module category structures are strict so that the above morphisms are identities.

A \emph{functor of right $\cC$-module categories} is given by a functor $F\colon \cM \to \cM'$ of the underlying linear categories, together with natural isomorphisms
\[
  F_2 = F_{2;m,U}\colon F(m) \stmryolt U \to F(m \stmryolt U) \quad (m \in \cM, U \in \cC)
\]
satisfying the standard compatibility conditions with structure morphisms of $\cM$ and $\cM'$.
If $\cM$ and $\cM'$ are C$^*$-$\cC$-module categories, a functor $(F,F_2)$ as above is a said to be a \emph{functor of right C$^*$-$\cC$-module categories} if it is a $*$-functor and the natural isomorphism $F_2$ is unitary.

\subsection{Vector spaces and algebras over \texorpdfstring{$\cC$}{CC}}

Let us review the correspondence between module categories and algebra objects for linear tensor categories.
Further details on the following construction can be found in~\cite{jp1}.

Let $\cC$ be a rigid tensor category.
We denote by $\Vect(\cC)$ the category whose objects are the contravariant linear functors $\cC \to \Vect$ and morphisms are natural transformations.
For $V \in \Vect(\cC)$, we call its values $V(X)$ on $X \in \cC$ the \emph{fibers} of $V$.

Due to the semisimplicity of $\cC$, it is possible to give a concrete description of $\Vect(\cC)$; the morphism space between objects $V$ and $W$ of $\Vect(\cC)$ can be decomposed into the algebraic direct product as
\[
  \Vect(\cC)(V,W) \cong \prod_{i \in \Irr(\cC)} \cL(V(U_i),W(U_i)),
\]
while the fibers can be written as
\[
  V(Y) \cong \bigoplus_{i \in \Irr(\cC) } V(U_i) \otimes \cC(Y,U_i).
\]
Under this presentation of $V$, its value on a morphism $\psi \in \cC(X,Y) = \cC^{\op}(Y^\op,X^\op)$ would then be
\[
  V(\psi) = \bigoplus_{i \in \Irr(\cC)} \id_{V(U_i)} \otimes \psi^\#
\]
with $\psi^\#$ denoting pre-composition with $\psi$.

The category $\Vect(\cC)$ becomes a monoidal category by introducing the monoidal product
\begin{equation}\label{eq:mon-prod-vec-obj-decomp}
  (V \otimes W)(X) =  \bigoplus_{i, j \in \Irr(\cC)} V(U_i) \otimes \cC(X, U_i \otimes U_j) \otimes W(U_j)
\end{equation}
together with naturally defined structure morphisms (see~\cite{jp1}).
Then the Yoneda embedding $X \to \cC( \cdot ,X) $ is a fully faithful monoidal functor from $\cC$ to $\Vect(\cC)$.
We regard $\cC$ as a monoidal subcategory of $\Vect(\cC)$ by means of this embedding.

An \emph{algebra object} in $\Vect(\cC)$ is given by $A \in \Vect(\cC)$ together with natural transformations $m\colon A \otimes A \to A$ and $i\colon 1_{\cC} \to A$ that make the algebra diagrams commute.
We denote by $\Alg(\cC)$ the category whose object are algebras in $\Vect(\cC)$ and the morphisms are algebra natural transformations.
In other words, the algebra objects in $\Vect(\cC)$ are exactly the lax tensor functors $\cC^{\op} \to \Vect$.

\begin{definition}
A \emph{pointed} right $\cC$-module category is a pair $(\cM, m)$, where $\cM$ is a right $\cC$-module category and $m$ is an object of $\cM$.
A $\cC$-module category is said to be \emph{cyclic} if it has a generating object $m$, in the sense that any object $X \in \cM$ is a direct summand of $m \stmryolt U$ for some $U \in \cC$.
By a \emph{pointed cyclic $\cC$-module category} we mean a pointed category $(\cM, m)$ such that $m$ is a generator of $\cM$.
\end{definition}

When the choice of $m$ is implicit, we just write $\cM$ instead of $(\cM, m)$, and $m_\cM = m$.

\begin{definition}
A \emph{functor of pointed right $\cC$-module categories} is given by a pair consisting of a $\cC$-module functor $F \colon \cM \to \cM'$ and an isomorphism $F_0 \colon m_{\cM'} \to F(m_\cM)$.
We denote by $\cMod_\cC(*)$ the category of pointed cyclic right $\cC$-module categories, with these functors as morphisms.
\end{definition}

\begin{theorem}[\cites{ostrik03,jp1}]\label{equivalgmod}
There is an equivalence of categories $\Alg(\cC) \simeq \cMod_\cC(*)$.
\end{theorem}

Let us sketch this correspondence.
Given $A \in \Alg(\cC)$, we get a right $\cC$-module category by taking the category $\cM_A$ of left $A$-module objects in $\cC$.
Concretely, we can start with objects $X_A \in \cM_A$ for $X \in \cC$ (corresponding to the left $A$-modules $A \otimes X$) and morphism sets
\[
  \mathcal{M}_A(X_A,Y_A) = A(X \otimes \bar{Y}),
\]
and then take the idempotent completion.
In this presentation the $\cC$-module structure is induced by $X_A \stmryolt U = (X \otimes U)_A$.

Conversely, given $\mathcal{M} \in \cMod_\cC(*)$, we define $A_\cM \in \Vect(\cC)$ by setting its fibers as
\[
  A_\cM(X) = \cM(m_\cM \stmryolt X, m_\cM),
\]
and its action on $\psi \in \cC(X,Y)$ as
\[
  A_\cM(\psi)\colon \cM(m_\cM \stmryolt Y, m_\cM) \to \cM(m_\cM \stmryolt X, m_\cM), \quad f \mapsto f (\id \otimes \psi) .
\]

Now, let us assume that $\cC$ is a rigid C$^*$-tensor category.

\begin{definition}[\cite{jp1}]\label{def:c-star-alg-obj}
An algebra $A \in \Alg(\cC)$ is said to be a \emph{$C^*$-algebra object} in $\cC$ when the corresponding pointed cyclic $\cC$-module category $\cM_A$ admits a compatible structure of C$^*$-category, and \emph{a $W^*$-category} when $\cM_A$ admits a compatible structure of $W^*$-category.
\end{definition}

See Appendix~\ref{app:intr-char-C-star-alg} for an intrinsic characterization of this concept.
When $A$ is a C$^*$-algebra object in the above sense, we frequently regard $A(U \otimes \bar U)$ as a C$^*$-algebra up to the isomorphism $A(U \otimes \bar U) \cong \cM_A(U_A, U_A)$, which is canonical.

\begin{definition}\label{hilbertspaces}
The category $\Hilb(\cC)$ is the subcategory of $\Vect(\cC)$ consisting of contravariant $*$-functors $\cC \to \Hilb$, and having uniformly bounded natural transformations as morphisms: for $H,K \in \Hilb(\cC)$,
\[
  \Hilb(\cC)(H,K) \cong \ell^\infty\mhyph\smashoperator[l]{\prod_{i \in \Irr(\cC)}} \cB(H(U_i),K(U_i)).
\]
\end{definition}

$\Hilb (\cC)$ becomes monoidal subcategory of $\Vect(\cC)$ by giving the following inner product to the fibers of $H \otimes K$.
On the space $H(U_i) \otimes \cC(X, U_i \otimes U_j) \otimes K(U_j)$, consider the Hermitian inner product characterized by
\[
  (\xi_2 \otimes \alpha_2 \otimes \eta_2, \xi_1 \otimes \alpha_1 \otimes \eta_1 ) 1_X = \frac{1}{d_i d_j} (\xi_2, \xi_1) (\eta_2, \eta_1) \alpha_1^* \alpha_2.
\]
Then $(H \otimes K)(X)$ has an inner product such that~\eqref{eq:mon-prod-vec-obj-decomp} gives an orthogonal decomposition.

\begin{example}
When $\cC = \Rep(G)$ for some compact quantum group $G$, the category $\Hilb(\Rep(G))$ is then the category of all unitary representations of $G$.
\end{example}

\begin{remark}\label{equivhilbindcat}
An object $H \in \Hilb(\cC)$ can be interpreted as the infinite direct sum $\bigoplus_i H(U_i) \otimes U_i$.
This way $\Hilb(\cC)$ can be identified with the ind-category $\ind(\cC)$ of $\cC$ as defined in~\cite{NY2}*{Section 2}.
\end{remark}

We will work with the following W$^*$-algebra object $\cB(H)$ for $H \in \Hilb(\cC)$~\cite{jp1}*{Example 10}.
The fiber is given by $\cB(H)(U) = \Hilb(\cC)(H \otimes U, H)$, with a natural algebra structure such that the associated module category $\cM_{\cB(H)}$ is generated by objects of the form $U_{\cB(H)}$ for $U \in \cC$ and the morphism spaces are given by
\[
  \cM_{\cB(H)}(U_{\cB(H)}, V_{\cB(H)}) = \Hilb(\cC)(H \otimes U, H \otimes V).
\]

Later, we will use the following analogue of the Gelfand--Naimark Theorem to construct injective envelopes of $C^*$-module categories.

\begin{theorem}[\cite{jp1}*{Theorem 4}]\label{thm:mod-cat-Gelfand-Naimark}
Every $C^*$-algebra object in $\Vect(\cC)$ admits a faithful representation into $\cB(H)$, for some $H \in \Hilb(\cC)$.
\end{theorem}

Here, a \emph{faithful representation} means a *-algebra morphism $A \to \cB(H)$ where the associated functor of module categories $\cM_A \to \cM_{\cB(H)}$ is faithful, or equivalently, the corresponding natural transformation $A(U) \to \cB(H)(U)$ is given by injective maps.

\begin{remark}
When $\cC$ is $\Rep(G)$ for some compact quantum group $G$, the above theorem reduces to the fact that any $G$-C$^*$-algebra can be equivariantly embedded into one of the form $\cR(\cB(H))$ for some infinite dimensional unitary representation $(H, U)$ of $G$.
\end{remark}

Recall that if $A$ and $B$ are C$^*$-algebra objects in $\Vect(\cC)$, then the spaces $A(X\bar{X})$ and $B(X \bar{X})$ are C$^*$-algebras for all $X \in \cC$.

\begin{definition}[\cite{jp1}]\label{cpmaps}
Let $A$ and $B$ be $C^*$-algebra objects in $\Vect(\cC)$.
A \emph{completely positive map} (or a \emph{cp map}) from $A$ to $B$ is a natural transformation of contravariant C$^*$-functors $\theta\colon A \to B$ for which the induced maps $\theta_{X \bar{X}}\colon A(X \bar{X}) \to B(X \bar{X})$, with $X \in \cC$, are positive.
If they are also unital, $\theta$ is called a \emph{unital completely positive map}, or a \emph{ucp map}.
\end{definition}

\begin{definition}[\cite{jp1}]\label{multiplier}
Let $(\cM, m)$ and $(\cM', m')$ be pointed C$^*$-$\cC$-module categories.
A \emph{multiplier} $\Theta\colon (\cM,m) \to (\cM',m')$ is a collection of linear maps
\[
  \Theta_{X,Y} \colon \cM(m \stmryolt X, m \stmryolt Y) \to \cM'(m' \stmryolt X, m' \stmryolt Y)
\]
for $X, Y \in \cC$ satisfying
\[
  \Theta_{X \otimes U, Z \otimes U} ( ((\id_m \stmryolt \phi) f (\id_m \stmryolt \psi)) \stmryolt \id_U) = ((\id_{m'} \stmryolt \phi) \Theta_{Y,W}(f) (\id_{m'} \stmryolt \psi)) \stmryolt \id_U
\]
for all $U, X, Y, W, Z \in \cC$, $\psi \in \cC(X,Y)$, $\phi \in \cC(W,Z)$, and $f \in \cM(m \stmryolt Y, m \stmryolt W)$.
A multiplier $\Theta$ for which $\Theta_{X,X}$ is positive for all $X$ is called a \emph{cp multiplier}.
It is called \emph{ucp multiplier} if $\Theta_{X,X}$ is ucp for $X$.
\end{definition}

\begin{proposition}[\cite{jp1}*{Proposition 7 and Corollary 5}]\label{prop:jp1-corr-betw-multip-nat-trans}
Given two $C^*$-algebra objects $A,B$ in $\Vect(\cC)$, the space of natural transformations between the contravariant functors $A$ and $B$ is in bijection with the space of multipliers from $\cM_A$ to $\cM_B$.
Under this bijection, the cp maps correspond to the cp multipliers, as the ucp maps correspond to the ucp multipliers.
\end{proposition}

Concretely, given a natural transformation $\theta\colon A \to B$, the corresponding multiplier is given by
\begin{equation}\label{eq:map-to-multiplier}
  \Theta_{V, W}(T) = \sum_{U, \alpha} ((\theta_U(\bar R_W^* (T \otimes \id_{\bar W}) v_\alpha) v_\alpha^* )\otimes \id_W) (\id_V \otimes R_W)
\end{equation}
where $U$ runs over $\Irr(\cC)$ and $(v_\alpha)_\alpha$ is an orthonormal basis of $\cC(U, V \otimes \bar W)$.

\subsection{Quantum groups}\label{sec:qg-prelim}

We follow the convention of~\cite{YDpaper} about quantum groups and their actions.
Let us briefly go over the notations.

Given a compact quantum group $G$, we denote the algebra of regular functions (matrix coefficients of finite dimensional unitary representations) by $\cO(G)$, which is a $*$-Hopf algebra.
We denote the reduced C$^*$-algebra of $G$ by $C(G)$, and its von Neumann algebraic closure with respect to the Haar state by $L^\infty(G)$.

A \emph{continuous action} of $G$ on a C$^*$-algebra $A$ is given by a nondegenerate and injective $*$-homomorphism $\alpha\colon A \to C(G) \otimes A$ satisfying $(\Delta \otimes \id) \alpha = (\id \otimes \alpha) \alpha$.
We say that $A$ is a \emph{$G$-C$^*$-algebra}.

Given a $G$-C$^*$-algebra $(A, \alpha)$, its \emph{regular subalgebra} $\cA$ is defined as the set of elements $a \in A$ such that $\alpha(a)$ belongs to the algebraic tensor product $\cO(G) \otimes A$.
This is a left $\cO(G)$-comodule algebra.

A measurable model, \emph{$G$-von Neumann algebra}, is given by a von Neumann algebra $M$ and a unital injective normal $*$-homomorphism $\alpha \colon M \to L^\infty(G) \vnotimes M$ satisfying the coassociativity condition $(\Delta \vnotimes \id) \alpha = (\id \vnotimes \alpha)\alpha$.
The regular subalgebra $\cM \subset M$ makes sense as above, and we denote its C$^*$-algebraic closure by $\cR(M)$.
Then $\cR(M)$ admits a continuous action of $G$.

\begin{example}\label{ex:adj-action}
Let $(H, U)$ be a finite dimensional unitary representation of $G$.
Under our convention, $H$ is a right $\cO(G)$-comodule, but it can be considered as a left comodule by the coaction map $\xi \mapsto U_{21}^* (1 \otimes \xi)$.
This becomes an equivariant right Hilbert $\C$-module by the inner product
\[
  \langle \xi, \eta \rangle_\C = (U_1 \eta, U_{1'} \xi)_H h(U_{2'}^* U_2) = (\rho_U^{-1} \eta, \xi)_H
\]
where $( \cdot, \cdot)_H$ denotes the original inner product on $H$, see~\cite{MR3933035}.
Then $\cB(H)$ admits the induced $G$-C$^*$-algebra structure, concretely given by the coaction $T \mapsto U^*_{2 1} T_2 U_{2 1}$.
Analogously, when $(H, U)$ is an infinite dimensional unitary representation, the same formula makes $\cB(H)$ a $G$-von Neumann algebra and $\cR(\cB(H))$ a $G$-C$^*$-algebra.
\end{example}

The finitely supported functions on the discrete dual $\hat{G}$, which is the direct sum of matrix algebras for the irreducible unitary representations of $G$, is denoted by $c_c(\hat{G})$.
Our convention of coproduct on this multiplier Hopf algebra satisfies the equalities
\begin{align}\label{eq:dual-hopf-alg-pairing}
  (\phi_{(1)}, f_1) (\phi_{(2)}, f_2) & = (\phi, f_1 f_2), & (\phi_1 \phi_2, f) & = (\phi_1, f_{(1)}) (\phi_2, f_{(2)}).
\end{align}
for the natural pairing between $\cO(G)$ and $c_c(\hat{G})$.

We denote the \emph{Drinfeld double} of $G$ by $D(G)$.
Its convolution algebra of functions $\cO_c(\hat D(G))$, is the tensor product coalgebra $c_c(\hat G) \otimes \cO(G)$, with the $*$-algebra structure induced by those of $\cO(G)$ and $c_c(\hat G)$ together with the exchange rule
\[
  (a_{(1)} \rhd \omega) a_{(2)} = a_{(1)} (\omega \lhd a_{(2)}) \qquad (\omega \in c_c(\hat G), a \in \cO(G))
\]
to make sense of $(\omega a)^* = a^* \omega^*$ and $(\omega_1 a_1) (\omega_2 a_2)$ inside the tensor product.

A \emph{Yetter--Drinfeld $G$-C$^*$-algebra} represents an action of $D(G)$ on a C$^*$-algebra.
Concretely, we take a $G$-C$^*$-algebra $(A, \alpha)$ together with a left action
\[
  \cO(G) \otimes \cA \to \cA, \quad f \otimes a \mapsto f \rhd a
\]
such that
\[
  \alpha(f \rhd a) = f_{(1)} a_{(1)} S(f_{(3)}) \otimes (f_{(2)} \rhd a_{(2)}).
\]
Maps between such algebras, compatible with the coactions and the actions of $\cO(G)$, are called \emph{$D(G)$-equivariant} or \emph{Yetter--Drinfeld $G$-equivariant} maps.

Let $A$ be a left $\cO(G)$-module algebra.
By duality, we have a homomorphism
\[
  \beta_A\colon A \to \smashoperator[l]{\prod_{i \in \Irr(G)}} A \otimes \cB(H_i), \quad a \mapsto \biggl(\sum_{k,l} (u^{(i)}_{kl} \triangleright a) \otimes m^{(i)}_{kl}\biggr)_i,
\]
which can be regarded as a right comodule algebra over the multiplier Hopf algebra representing $\hat G$ with the convention of~\eqref{eq:dual-hopf-alg-pairing}.

This construction makes sense as a unitary coaction of $\ell^\infty(\hat G)$ when $A$ is a Yetter--Drinfeld $G$-C$^*$-algebra.
Moverover, when we consider an action of $\hat G$ on a von Neumann algebra $N$, the target of coaction map becomes
\[
  \ell^\infty \mhyph\smashoperator[l]{\prod_{i \in \Irr(G)}} N \otimes \cB(H_i) \simeq N \vnotimes \ell^\infty(\hat{G}).
\]

\medskip
By an analogue of the Tannaka--Krein--Woronowicz duality, the category of unital $G$-C$^*$-algebras is equivalent to the category of right C$^*$-$(\Rep(G))$-module categories~\citelist{\cite{ydk1}\cite{n1}}.

Concretely, given a unital $G$-C$^*$-algebra $B$, one takes the category $\cD_B$ of finitely generated projective $G$-equivariant right Hilbert modules over $B$.
Thus, an object of $\cD_B$ is a right Hilbert module $E_B$ with a left coaction  $\delta$ of $C(G)$, such that the action of $B$ is equivariant.
Given an object $U$ of $\Rep(G)$, its right action $E_B \stmryolt U$ is represented by the equivariant right Hilbert module $H_U \otimes E_B$, where the underlying left comodule is given as the tensor product of $E_B$ and the left comodule $H_U$ as explained in Example~\ref{ex:adj-action}.
Explicitly, the new coaction on $E_B \stmryolt U$ is given by
\[
  \xi \otimes a \mapsto \left( U_{21}^*(1 \otimes \xi \otimes 1)\right)\delta(a)_{13}.
\]

Conversely, given a right C$^*$-$(\Rep(G))$-module category $\cD$ and an object $X \in \cD$, we take the left $\cO(G)$-comodule
\[
  \cB_{\cD, X} = \bigoplus_{i \in \Irr(G)} \bar H_i \otimes \cD(X, X \stmryolt U_i),
\]
which admits an associative product from irreducible decomposition of monoidal products.
Together with the involution coming from duality of representations, we obtain a pre C$^*$-algebra which admits a canonical completion supporting a coaction of $C(G)$.

\begin{remark}\label{rem:leftcoactionsandrightmodules}
The formula above explains why left $C(G)$-comodule structures give rise to right $\Rep(G)$-module categories.
Indeed, given two finite dimensional unitary representations $(H_U, U)$ and $(H_V,V)$ of $G$, and given $\xi \in H_U$, $\eta \in H_V$ and $a \in E_B$, the $C(G)$-coaction on $\xi \otimes \eta \otimes a \in (H_V \otimes E_B) \stmryolt U$ is given by
\[
  \xi \otimes \eta \otimes a \mapsto U_{21}^* V_{31}^* (a_{(1)} \otimes \xi \otimes \eta \otimes a_{(2)}),
\]
where $a_{(1)} \otimes a_{(2)} = \delta(a)$.
Flipping the second and third legs, we obtain
\[
  U_{21}^* V_{31}^* ( a_{(1)} \otimes \xi \otimes \eta \otimes a_{(2)}) \mapsto (V \otimes U)_{21}^* (a_{(1)} \otimes (\eta \otimes \xi) \otimes a_{(2)}).
\]
This computation shows that the flip map $H_U \otimes H_V \to H_V \otimes H_U$ induces an equivariant isomorphism $(H_V \otimes E_B) \stmryolt U \simeq H_{V \otimes U} \otimes E_B$.
\end{remark}

The generalization of module categories to the nonunital setting is given by multiplier module categories~\cite{AV1}.

\section{Injective envelopes for \texorpdfstring{$D(G)$-C$^*$}{D(G)-C*}-algebras}\label{sec:inj-env-YD-G-algs}

In this section we are going to establish the existence of injective envelopes for unital Yetter--Drinfeld $G$-C$^*$-algebras, closely following~\cite{HHN1}*{Section 2}.

Take a unitary representation $(H, U)$ of $G$ on $H$, and as before consider the coaction
\[
  \alpha_H(T) = U_{21}^* (1 \otimes T) U_{21}
\]
on $\cB(H)$.
Then $\cB(H) \vnotimes \ell^\infty (\hat{G})$ becomes a $G$-von Neumann algebra with coaction
\begin{equation*}
  \beta\colon \cB(H) \vnotimes \ell^\infty (\hat{G}) \to L^\infty(G) \vnotimes \cB(H) \vnotimes \ell^\infty (\hat{G}), \quad x \mapsto W_{31}^* (\alpha_H \otimes \id ) (x) W_{31},
\end{equation*}
where $W = W_G$ is the multiplicative unitary.
Together with the $\hat G$-von Neumann algebra structure given by $\id \otimes \hat{\Delta}$, we get a Yetter--Drinfeld von Neumann $G$-algebra structure on $\cB(H) \vnotimes \ell^\infty (\hat{G})$.

\begin{proposition}[~\cite{HHN1}, Section 2]
Under the above setting, the $G$-C$^*$-algebra $\cR(\cB(H) \otimes \ell^\infty (\hat{G}))$ admits a structure of Yetter--Drinfeld $G$-C$^*$-algebra.
\end{proposition}

Now let us define injective envelopes for unital Yetter--Drinfeld $G$-C$^*$-algebras, in the quite standard way.
In this section, by equivariance we always mean equivariance with respect to Yetter--Drinfeld structures.

\begin{definition}
A unital Yetter--Drinfeld $G$-C$^*$-algebra $A$ is said to be \emph{injective} if for any equivariant complete isometric map $\phi \colon B \to C$  of unital Yetter--Drinfeld $G$-C$^*$-algebras and any equivariant ucp map $\psi\colon B \to A$, there is an equivariant ucp map $\tilde\psi\colon C \to A$ that extends $\psi$ along $\phi$, i.e., $\tilde{\psi} \circ \phi = \psi$.
\end{definition}

\begin{definition}
An \emph{injective envelope} of a unital Yetter--Drinfeld $G$-C$^*$-algebra $A$ is given by an injective Yetter--Drinfeld $G$-C$^*$-algebra $\cI$ and an equivariant complete isometry $\phi \colon A \to \cI$ which is \emph{essential}, i.e., for any equivariant ucp map $\psi \colon \cI \to B$, $\psi$ is completely isometric if and only if $\psi \phi$ is.
\end{definition}

The first step in proving the existence of injective envelopes is to show the existence of enough injective objects.
This is accomplished by the following proposition.

\begin{proposition}\label{prop:R-BH-linftyhatG-inj}
Let $(H, U)$ be a unitary representation of $G$.
The Yetter--Drinfeld $G$-C$^*$-algebra $\cR(\cB(H) \vnotimes \ell^\infty (\hat{G}))$ is injective.
\end{proposition}

\begin{proof}
The proof is essentially the same as~\cite{HHN1}*{Corollary 2.6}.
By~\cite{HHN1}*{Proposition 2.5}, for any Yetter--Drinfeld $G$-C$^*$-algebra $A$ there is a bijective correspondence between the $G$-equivariant completely bounded maps $\phi\colon A \to \cB(H)$ and $D(G)$-equivariant cb maps $P\colon A \to \cB(H) \vnotimes \ell^\infty(\hat G)$.
Moreover, the image of such map $P$ would be in $\cR(\cB(H) \vnotimes \ell^\infty(\hat G))$.

By the averaging argument (see~\cite{HHN1}, Lemma 2.10), the injectivity of $\cB(H)$ as an operator system implies its injectivity as a $G$-operator system.
This implies the injectivity of $\cR(\cB(H) \vnotimes \ell^\infty(\hat G))$ as a Yetter--Drinfeld $G$-C$^*$-algebra.
\end{proof}

Let $A$ be a Yetter--Drinfeld $G$-C$^*$ algebra, and let us take a faithful representation $\pi\colon A \to \cB(H)$ together with a covariant unitary representation $U$ of $G$ on $H$.
For example, we may start from a faithful nondegenerate representation $\pi_0 \colon A \to B(H_0)$, and take
\begin{align*}
  H & = L^2(G) \otimes H_0, & \pi & = (\lambda \otimes \pi_0) \alpha, & U & = W_{1 3} \in \cM(\cK(L^2(G) \otimes H_0) \otimes C(G)).
\end{align*}

\begin{definition}
Under the above setting, the \emph{Poisson integral} of $\pi$ is defined as the $D(G)$-equivariant embedding map
\[
  P_{\pi}\colon A \to \cB(H) \vnotimes \ell^\infty (\hat{G}), \quad a \mapsto (\pi \otimes \id) \beta_A(a).
\]

\end{definition}

\begin{theorem}\label{thm:inj-env-YD-G-alg}
Every unital Yetter--Drinfeld $G$-C$^*$-algebra $A$ admits an injective envelope.
\end{theorem}

\begin{proof}
Let us fix $(\pi, H)$ as above.
Consider the semigroup $S$ of equivariant ucp maps $f$ from $\cR(\cB(H) \otimes \ell^\infty(\hat{G}))$ to itself such that $f P_\pi = P_\pi$.
This is closed with respect to the topology of pointwise convergence for the weak$^*$-topology.
By Proposition~\ref{minimalidemp}, there is a unique minimal idempotent in $\phi_0 \in S$.

The rest is quite standard, as follows.

Put $\cI(A) = \Img(\phi_0)$, and endow it with the Choi--Effros product coming from $\phi_0$.
By construction, $\cI(A)$ is injective as a Yetter--Drinfeld $G$-C$^*$-algebra, and we have an embedding $A \to \cI(A)$ given by $P_\pi$.
We are going to show that $\cI(A)$ is an injective envelope of $A$.

First, we claim that the only ucp $D(G)$-equivariant map $\psi\colon \cI(A) \to \cI(A)$ restricting to $\id_A$ on $A$ is the identity map.
Otherwise $\psi \phi_0$ would be an element of $S$ that is strictly below $\phi_0$, contradicting the minimality of $\phi_0$.

Now, let us take a completely isometric $D(G)$-equivariant ucp map $\psi\colon\cI(A) \to B$ such that $\psi P_\pi$ is completely isometric.
Without losing generality we may assume that $B$ is injective.
By the injectivity of $\cI(A)$, there is an equivariant ucp map $\eta\colon B \to \cI(A)$.
By the above claim implies that $\eta \psi$ must be the identity, which implies that $\psi$ is completely isometric.
This shows that $\cI(A)$ is an essential extension of $A$.
\end{proof}

\begin{remark}
Theorem~\ref{thm:inj-env-YD-G-alg} would extend to more general locally compact quantum groups once we have an analogue of Proposition~\ref{prop:R-BH-linftyhatG-inj}.
Note that~\cite{MR2985658} already establishes the existence of injective envelopes for the category of operator spaces with von Neumann algebraic coaction of a Hopf--von Neumann algebra $M$, such that the predual Banach algebra $M_*$ has a bounded approximate unit with bound $1$.
For us, the relevant Hopf--von Neumann algebra is $M = L^\infty(D(G))$, so that $M_*$ is a completion of $\cO_c(\hat D(G))$.
Taking an increasing net of finite sets of $\Irr(G)$, we obtain an approximate unit of $M_*$ of the form $(z_i \otimes 1)_i$, where $z_i$ is a central projection of $c_c(\hat G)$.
In this regard Theorem~\ref{thm:inj-env-YD-G-alg} could be regarded as a refinement of~\cite{MR2985658}*{Theorem 2.7} to the category of \emph{continuous} actions of $D(G)$.
\end{remark}

\section{Injectivity for module categories}\label{sec:inj-mod-cat}

As before, let $\cC$ be a rigid C$^*$-tensor category.
Our goal here is to establish that the theory of injectivity work well for $\cC$-module categories, which will form a basis of our later analysis for module categories with additional structures.

Let us fix an object  $U$ of $\cC$.
Let us denote by $\cC_U$ the $C^*$-$\cC$-module category defined as (the idempotent completion of) the category with objects $V \in \cC$, and morphism sets
\[
  \cC_U(V,W) = \cC(U\otimes V,U \otimes W).
\]
In other words, this is simply the category $\cC$ as a right module category over itself, but with shifted base point $U$.

Let us call $H \in \Hilb(\cC)$ \emph{finite-dimensional} when it is isomorphic to the image under the Yoneda embedding of some object $U \in \cC$.
A Hilbert space object is finite dimensional if and only if it has finite dimensional fibers and finite support: $H(V) = 0$ for all but finitely many $V \in \Irr(\cC)$.

As we show in the next proposition, the categories $\cC_U$ can be used as a model of $\cB(H)$ for finite dimensional Hilbert space objects $H$.

\begin{proposition}\label{prop:mod-cat-for-end-alg-obj}
Let $H \in \Hilb(\cC)$ be finite dimensional, corresponding to an object $U \in \cC$.
Then the pointed cyclic $C^*$-$\cC$-module category $\cM_{\cB(H)}$ is isomorphic to $\cC_{U}$.
\end{proposition}

\begin{proof}
Let us write $\tilde V$ for the image of Yoneda embedding of $V \in \cC$, so that we have $H \simeq \tilde U$.
Since $V \mapsto \tilde V$ is a monoidal functor, we have
\[
  (H \otimes \tilde V) (Z) \simeq \cC(Z,U \otimes V).
\]

Now, recall that we have
\begin{equation*}
  \cM_{\cB(H)}(V,W) = \cB(H)(V \otimes \bar{W})  \simeq \ell^\infty\mhyph\smashoperator[l]{\prod_{Z \in \Irr(\cC)}} \cB\left((H \otimes \tilde V)(Z), (H \otimes \tilde W)(Z) \right).
\end{equation*}
We thus have
\[
  \cM_{\cB(H)}(V,W) \simeq \ell^\infty\mhyph\smashoperator[l]{\prod_{Z \in \Irr(\cC)}} \cB(\cC(Z,U \otimes V), \cC(Z,U \otimes W)).
\]
From this we see that an element $\theta \in \cM_{\cB(H)}(V,W)$ is the same thing as a bounded family of linear operators
\[
  \theta_Z\colon \cC(Z, U \otimes V) \to \cC(Z, U \otimes W) \quad (Z \in \Irr(\cC)).
\]
This is precisely a natural transformation between the Yoneda embeddings of $U \otimes V$ and $U \otimes W$, hence given by a morphism $U \otimes V \to U \otimes W$.
We thus obtained an isomorphism $\cM_{\cB(H)}(V,W) \simeq \cC_U(V, W)$.
\end{proof}

\begin{definition}\label{def:mod-cat-inj}
A pointed C$^*$-$\cC$-module category $(\cM, m)$ is said to be \emph{injective} when, for any ucp-multiplier  $\Phi \colon (\cN, n) \to (\cM, m)$ and another ucp-multiplier $\Psi\colon (\cN, n) \to (\cN', n')$ such that the maps
\[
  \Psi_{V, W} \colon \cN(n \stmryolt V, n \stmryolt W) \to \cN'(n' \stmryolt V, n' \stmryolt W)
\]
are completely isometric for $V, W \in \cC$, there exists a ucp-multiplier $\tilde\Phi \colon (\cN', n') \to (\cM, m)$ such that $\tilde \Phi \Psi = \Phi$.
\end{definition}

When the generator $m$ is understood from the context we also say $\cM$ is injective.
We will later see that this definition is independent of the choice of $m$ for cyclic module categories.

In the following lemmas, $(\cM, m)$ is a fixed pointed C$^*$-$\cC$-module category.

\begin{lemma}\label{lem:choi-mat-gen}
Let $(\cM', m')$ be another pointed C$^*$-$\cC$-module category.
Then there is a bijective correspondence between the cp-multipliers $P\colon (\cM, m) \to (\cM', m' \stmryolt U)$ and the cp-multipliers $Q\colon (\cM, m \stmryolt \bar U) \to (\cM', m')$.
Under this correspondence, a ucp-multiplier $P$ corresponds to a ucp-multiplier $Q$ satisfying
\begin{equation}\label{eq:index-property}
  Q_{U \otimes V, U \otimes V}(\id_m \stmryolt (R_U R_U^* \otimes \id_V)) = d_U^{-1} \id_{m' \stmryolt U \otimes V}
\end{equation}
for all $V \in \cC$.
\end{lemma}

\begin{proof}
Let $P \colon (\cM, m) \to (\cM', m' \stmryolt U)$ be a cp-multiplier, given by the maps
\[
  P_{V, W} \colon \cM(m \stmryolt V, m \stmryolt W) \to \cM'(m' \stmryolt U \otimes V, m' \stmryolt U \otimes W).
\]
Then we get a multiplier $Q \colon (\cM, m \stmryolt \bar U) \to (\cM', m')$ by setting
\begin{align*}
  Q_{V, W}\colon \cM(m \stmryolt \bar U \otimes V, m \stmryolt \bar U \otimes W) & \to \cM'(m' \stmryolt V, m' \stmryolt W), \\ T & \mapsto d_U^{-1} (\id_{m'} \stmryolt \bar R_U^* \otimes \id_W) P_{\bar U \otimes V, \bar U \otimes W}(T) (\id_{m'} \stmryolt \bar R_U \otimes \id_V).
\end{align*}
If $V = W$ and $T$ is positive, then $P_{\bar{U} \otimes V, \bar{U} \otimes V}(T)$ is positive by assumption, and since $d_U^{-1/2} \bar{R}_U$ is an isometry, we conclude that $Q_{V,V}(T)$ is positive.
This shows that $Q$ is completely positive.
Moreover, when $P$ is unital, $Q$ is unital by $\bar R_U^* R_U = d_U$, and~\eqref{eq:index-property} follows from $\cC$-modularity.

In the other direction, given a cp-multiplier $Q \colon (\cM, m \stmryolt \bar U) \to (\cM', m')$, we get a multiplier $P\colon (\cM, m) \to (\cM', m' \stmryolt U)$ by
\[
  P_{V, W}(T) = d_U^{-1} Q_{U \otimes V, U \otimes W}((\id_m \stmryolt R_U \otimes \id_W) T (\id_m \stmryolt R_U^* \otimes \id_V)).
\]
This is again completely positive by the complete positivity of $Q$.
Moreover, when $Q$ satisfies~\eqref{eq:index-property}, the unitality of $P$ is a direct consequence of the above definition.

It remains to check that these constructions are inverse to each other.
Let us start from a multiplier $P$ as above, and let $Q$ be the corresponding one.
Let $P'\colon (\cM, m) \to (\cM', m' \stmryolt U)$ denote the multiplier we obtain from $Q$.
We then need to check that $P'$ is equal to $P$.
Expanding the definitions, we have
\begin{multline*}
  P'(T) =\\ (\id_{m'} \stmryolt \bar R_U^* \otimes \id_{U \otimes W}) P_{\bar U \otimes U \otimes V, \bar U \otimes U \otimes W} \bigl( (\id_m \stmryolt R_U \otimes \id_W) T (\id_m \stmryolt R_U^* \otimes \id_V) \bigr) (\id_{m'} \stmryolt \bar R_U \otimes \id_{U \otimes V}).
\end{multline*}
Using
\begin{multline*}
  P_{\bar U \otimes U \otimes V, \bar U \otimes U \otimes W} \bigl( (\id_m \stmryolt R_U \otimes \id_W) T (\id_m \stmryolt R_U^* \otimes \id_V) \bigr) =\\ (\id_{m'} \stmryolt \id_U \otimes R_U \otimes \id_W) P_{V, W}( T ) (\id_{m'} \stmryolt \id_U \otimes R_U^* \otimes \id_V)
\end{multline*}
and the conjugate equations for $U$, we see that $P'(T) = P(T)$ indeed holds.
This can also be seen from the graphical calculus for multipliers established in~\cite{jp1}.
\end{proof}

\begin{remark}
In fact, the condition~\eqref{eq:index-property} for all $V \in \cC$ implies that $Q$ is a unital multiplier.
To see this, given $W \in \cC$, take $V = \bar U \otimes W$ and consider the operation
\[
  T \mapsto (\id_{m'} \stmryolt \bar R_U^* \otimes \id_W) T (\id_{m'} \stmryolt \bar R_U \otimes \id_W)
\]
on both sides of~\eqref{eq:index-property}.
The left hand side gives $Q_{W,W}(\id_{m \stmryolt W})$ by the conjugate equations, while the right hand side gives $\id_{m' \stmryolt W}$ by $\bar R_U^* \bar R_U = d_U$.
\end{remark}

\begin{lemma}[cf.~\cite{HHN1}*{Proposition 4.11}]\label{lem:cp-map-to-C-and-state}
There is a bijective correspondence between the ucp-multipliers $(\cM, m) \to (\cC, 1_\cC)$ and the states on the C$^*$-algebra $\cM(m)$.
\end{lemma}

\begin{proof}
Suppose we are given a ucp-multiplier $P\colon (\cM, m) \to (\cC, 1_\cC)$.

Given $V \in \cC$, define a state $\omega_V$ on $\cM(m \stmryolt V)$ by $\omega_V(T) = \tr_V(P_{V,V}(T))$.
By the multiplier property, we have
\begin{equation}\label{eq:omega-V-from-P-1-1}
  \omega_V(T) = d_V^{-1} P_{1,1}((\id_m \stmryolt \bar R_V^*) (T \stmryolt \id_{\bar V}) (\id_m \stmryolt \bar R_V)).
\end{equation}

We claim that the state $\omega_1 = P_{1,1}$ determines both $\omega_V$ and $P_{V,V}$, which implies that the correspondence from ucp-multipliers to states is one-to-one.

First,~\eqref{eq:omega-V-from-P-1-1} shows that $\omega_V$ can be written in terms of $\omega_1$.

Observe that $\cC(V)$ embeds into $\cM(m \stmryolt V)$ by $T \mapsto \id_m \stmryolt T$, and the restriction of $\omega_V$ to $\cC(V)$ agrees with $\tr_V$.
Using the spherical structure, we see that $\cC(V)$ is in the centralizer of $\omega_V$, implying the existence of a unique state-preserving conditional expectation
\[
  E_V \colon \cM(m \stmryolt V) \to \cC(V), \quad \tr_V \circ E_V = \omega_V.
\]
By the uniqueness, we see that $P_{V,V} = E_V$, which shows that $E_V$ is indeed determined by $\omega_V$, hence by $\omega_1$.

Conversely, let $\omega\colon \cM(m) \to \C$ be a state.
Composing the conditional expectation
\[
  \cM(m \stmryolt V) \to \cM(m), \quad T \mapsto d_V^{-1} (\id_m \stmryolt \bar R_V^*) (T \stmryolt \id_{\bar V}) (\id_m \stmryolt \bar R_V)
\]
with $\omega$, we obtain a state $\omega_V$ on $\cM(m \stmryolt V)$.

Again $\cC(V)$ sits in the centralizer of $\omega$.
Setting $P_{V,V} \colon \cM(m \stmryolt V) \to \cC(V)$ to be the conditional expectation for $\omega_V$, we obtain a ucp-multiplier satisfying $\omega = P_{1,1}$.
\end{proof}

\begin{lemma}
Let $\omega$ be a state on $\cM(m \stmryolt \bar U)$, and $Q\colon (\cM, m \stmryolt \bar U) \to (\cC, 1)$ be the ucp-multiplier corresponding to $\omega$ given by Lemma~\ref{lem:cp-map-to-C-and-state}.
Then $Q$ satisfies~\eqref{eq:index-property} if and only if $\omega$ restricts to the normalized categorical trace on $\cC(\bar U)$.
\end{lemma}

\begin{proof}
Let $\omega$ be a state on $\cM(m \stmryolt \bar U)$.
Recall that from $\omega$ we can construct state $\omega_V$ on the C$^*$-algebra $\cM(m \stmryolt (\bar{U} \otimes V))$.
Let $Q$ be the corresponding ucp-multiplier.
Then, to check~\eqref{eq:index-property}, it is enough to check
\[
  \omega_{U \otimes V}((\id_m \stmryolt R_U R_U^* \otimes \id_V) (\id_{m \stmryolt \bar U} \stmryolt S)) = d_U^{-1} \tr_{U \otimes V}(S) \quad (S \in \cC(U \otimes V))
\]
as $Q_{U \otimes V,U \otimes V}$ is the unique conditional expectation $\cM(m \stmryolt( \bar{U} \otimes U \otimes V)) \to \cC(U \otimes V)$ such that $\tr_{U \otimes V} \circ Q_{U \otimes V, U \otimes V} = \omega_{U \otimes V, U \otimes V}$.
By the definition of $\omega_{U \otimes V}$, the left hand side of the above is equal to
\[
  d_U^{-1} \omega\left( \id_m \stmryolt  (R_U^* \otimes \id_{\bar U}) (\id_{\bar U} \otimes S' \otimes \id_{\bar U}) (\id_{\bar U} \otimes \bar R_U) \right)
\]
for $S' = (\id \otimes \tr_V)(S)$.
If $\omega$ restricts to $\tr_{\bar U}$, by a standard sphericity argument we see that this is indeed equal to $d_U^{-1} \tr_{U \otimes V}(S)$.

Conversely, suppose that we know~\eqref{eq:index-property}.
We then have
\[
  d_U Q_{U,U}(\id_m \stmryolt (R_U R_U^* (\id_{\bar U} \otimes S))) = S
\]
for any $S \in \cC(U)$.
From the form of state $\omega_U$, we obtain
\[
  \omega\left( \id_m \stmryolt \left ( (R_U^* \otimes \id_{\bar U}) (\id_{\bar U} \otimes S T \otimes \id_{\bar U}) (\id_{\bar U} \otimes \bar R_U)  \right ) \right) = \tr_U(S T) \quad (T \in \cC(U)).
\]
Again by standard sphericity argument we obtain $\omega(S^\vee) = \tr_{\bar U}(S^\vee)$ for
\[
  S^\vee = (R_U^* \otimes \id_{\bar U}) (\id_{\bar U} \otimes S \otimes \id_{\bar U}) (\id_{\bar U} \otimes \bar R_U).
\]
As such $S^\vee$ exhaust $\cC(\bar U)$, we obtain the claim.
\end{proof}

We now move towards proving that the category $\cMod_\cC(*)$ has enough injective objects.
First we show that objects of a particular type are injective.

\begin{lemma}\label{lemma:C-1-injective}
The object $(\cC, 1_\cC)$ is injective in $\cMod_\cC(*)$.
\end{lemma}

\begin{proof}
For any $(\cM,m) \in \cMod_\cC(*)$, the ucp-multipliers $(\cM,m) \to (\cC,1_\cC)$ are completely determined by the induced states on $\cM(m)$.
Moreover, a completely isometric multiplier $(\cM,m) \to (\cM',m')$ induces a (complete) isometric map of unital C$^*$-algebras $\cM(m) \to \cM'(m')$.
Then the claim follows from the Hahn--Banach theorem.
\end{proof}

\begin{proposition}\label{inj-indep-base-pt}
Let $(\cM, m)$ be an injective pointed $\cC$-module category.
For any $U \in \cC$ and a direct summand $m'$ of $m \stmryolt U$, the pointed $\cC$-module category $(\cM, m')$ is also injective.
\end{proposition}

\begin{proof}
Lemma~\ref{lem:choi-mat-gen} shows that $(\cM, m \stmryolt U)$ is injective.
Thus, it is enough to prove the assertion when $U = 1$.
Let us take the projection $p \in \cM(m)$ corresponding to $m'$, and let $m''$ be the summand corresponding to $1-p$.

Suppose that $\Phi\colon (\cN, n) \to (\cM, m')$ is a ucp-multiplier, and $\Psi\colon (\cN, n) \to (\cN',n')$ is a complete isometric multiplier as in Definition~\ref{def:mod-cat-inj}.
Take a state $\omega \colon \cN(n) \to \C$, and let $P \colon \cN \to \cC$ be the corresponding ucp-multiplier.
Composing this with the ucp-multiplier $Q \colon \cC \to (\cM, m'')$ given by
\[
  \cC(V, W) \to \cM(m'' \stmryolt V, m'' \stmryolt W), \quad T \mapsto (1-p) \stmryolt T,
\]
we obtain a ucp-multiplier $\Phi'\colon (\cN,n) \to (\cM, m'')$.
By taking the direct sum $\Phi \oplus \Phi'$, we get a ucp multiplier $(\cN, n) \to (\cM, m)$.
Then the injectivity of $(\cM, m)$ gives an ucp extension $\tilde\Psi'\colon (\cN',n') \to (\cM, m)$.
Then the maps $\tilde\Psi_{V,W}(T) = (p \stmryolt \id_W) \tilde\Psi'_{V,W}(T) (p \stmryolt \id_V)$ give a desired ucp extension $(\cN', n') \to (\cM, m')$.
\end{proof}

\begin{proposition}\label{C-U-injective}
Let $H \in \Hilb(\cC)$ be a finite dimensional object.
Then $\cM_{\cB(H)}$ is an injective pointed $\cC$-module category.
\end{proposition}
\begin{proof}
Let $H$ be a finite dimensional Hilbert space object.
By Proposition~\ref{prop:mod-cat-for-end-alg-obj}, there is $U \in \cC$ such that $\cM_{\cB(H)} \simeq \cC_U \simeq (\cC,U)$.
Since $(\cC,1_\cC)$ is injective by Lemma~\ref{lemma:C-1-injective}, $(\cC,U)$ is injective by Proposition~\ref{inj-indep-base-pt}.
\end{proof}

Now, following Arveson's proof for the injectivity of $\cB(H)$ for Hilbert spaces, we obtain the following.

\begin{theorem}\label{thm:arveson-type-thm-for-mod-cat}
For any $H \in \Hilb(\cC)$, the C$^*$-$\cC$-module category $\cM_{\cB(H)}$ is injective.
\end{theorem}

\begin{proof}
Let $A$ and $B$ be $C^*$-algebra objects, and suppose that $\iota\colon A \to B$ is a completely isometric multiplier and $\psi\colon A \to \cB(H)$ is a ucp multiplier.
We then need to construct an extension $\tilde \psi \colon B \to \cB(H)$.

Take an increasing net $(H_\lambda)_{\lambda \in \Lambda}$ of finite dimensional Hilbert space objects such that $H_\lambda(X) \subset H(X)$ and $H(X) = \varinjlim_{\lambda \in \Lambda} H_\lambda(X)$ for all $X \in \cC$.
Let us denote the orthogonal projections $H(X) \to H_\lambda(X)$ by $p_{\lambda,X}$.
Then we get a morphism $p_\lambda\colon H \to H_\lambda$ by
\[
  (p_{\lambda,X})_{X \in \Irr(\cC)} \in \ell^\infty\mhyph\smashoperator[l]{\prod_{X \in \Irr(\cC)}} \cB(H(X),H_\lambda(X)) \simeq \Hilb(\cC)(H,H_\lambda).
\]
The maps $\Ad(p_\lambda)\colon \cB(H_\lambda) \to \cB(H)$ and $\Ad(p_\lambda^*)\colon \cB(H) \to \cB(H_\lambda)$ are cp and ucp multipliers, respectively~\cite{jp1}*{Lemma 4.27}.

For fixed $\lambda$, since $\Ad(p_\lambda^*) \circ \psi$ is a ucp map from $A$ to $\cB(H_\lambda)$, we obtain a ucp extension $\tilde\psi_\lambda \colon B \to \cB(H_\lambda)$ by Proposition~\ref{C-U-injective}.
Then $\psi_\lambda = \Ad(p_\lambda) \circ \tilde{\psi}_\lambda$ is a completely contractive positive map from $B$ to $\cB(H)$.
We claim that a limit of the family $(\psi_\lambda)_\lambda$ is a desired ucp extension of $\psi$.

For each $V \in \cC$, the map
\[
  \tilde\psi_{\lambda, V}\colon B(V) \to \cB(H)(V) \simeq \ell^\infty\mhyph\smashoperator[l]{\prod_{X \in \Irr(\cC)}} \cB((H \otimes V)(X),H(X))
\]
is completely contractive.
By passing to a subnet, we may assume that $\psi_V = \lim_\lambda \psi_{\lambda, V}$ exists as a complete contraction for each irreducible $V$, where we consider the topology of pointwise convergence with respect to the weak operator topology on $\cB(H)(V)$ up to the above identification.

Now, collecting $\tilde\psi_V$ as above for irreducible $V$, we get a natural transformation $\tilde\psi\colon B \to \cB(H)$, and by Proposition~\ref{prop:jp1-corr-betw-multip-nat-trans} a multiplier $\tilde\Psi\colon \cM_B \to \cM_{\cB(H)}$.
Let $(\tilde\Psi_{\lambda, V, W})_{V, W}$ be the cp multiplier corresponding to $\psi_\lambda$.
We then claim that $\tilde\Psi_{V, W}(T) = \lim_\lambda \tilde\Psi_{\lambda, V, W}(T)$ for the weak operator topology.
This follows from the correspondence~\eqref{eq:map-to-multiplier} and $\tilde\psi_U(S) \otimes \id_W = \lim_\lambda \tilde\psi_{\lambda, U}(S) \otimes \id_W$ for $S \in B(U)$.

Then, setting $V = W$, we obtain the complete positivity of $\tilde\Psi$.
Moreover, $\id_{V_B} \in \cM_B(V_B, V_B)$ is the image of $\id_{V_A} \in \cM_A(V_A, V_A)$, so we have
\[
  \tilde\Psi_{V,V}(\id_{V_B}) = \lim (p_\lambda^* p_\lambda)_V = \id_{H(V)}.
\]
This proves that $\tilde\Psi$ is unital, hence $\tilde\psi$ is a ucp extension of $\psi$.
\end{proof}

\begin{definition}
A functor $F\colon \cM \to \cN$ between C$^*$-$\cC$-module is said to be an embedding if it is faithful and norm-closed at the level of morphism spaces.
\end{definition}

\begin{definition}\label{def:inj-envelope-mod-cat}
An \emph{injective envelope} of a C$^*$-$\cC$-module category $(\cM, m)$ is an injective $\cC$-module category $(\cI, m')$ endowed with a faithful module functor $\iota \colon (\cM, m) \to (\cI, m')$ such that $\Id_\cI$ is the only ucp extension of $\iota$.
\end{definition}

\begin{theorem}\label{thm:injective-envelopes-mod-cat}
Every cyclic C$^*$-$\cC$-module category $(\cM, m)$ has an injective envelope.
\end{theorem}

\begin{proof}
As we only deal cyclic module categories in this proof, for simplicity we write $\cM$ instead for $(\cM,m)$, and similarly for other cyclic module categories in this proof.

Theorems~\ref{thm:mod-cat-Gelfand-Naimark} and~\ref{thm:arveson-type-thm-for-mod-cat} imply that $\cM$ embeds into an injective $\cC$-module W$^*$-category $\cN$, namely $\cN = \cM_{\cB(H)}$ for some $H \in \Hilb(\cC)$.
Define $i \colon \cM \to \cN$ to be this embedding.
Consider the semigroup of ucp multipliers
\[
  S = \{ \Phi \colon \cN \to \cN \mid \Phi i = i \}.
\]
This can be identified with a weak$^*$-closed convex set in the dual Banach space
\[
  \ell^\infty\mhyph\smashoperator[l]{\prod_{U,V \in \Irr(\cC)}} \cN(U,V) = \Biggl[ \ell^1\mhyph \smashoperator[l]{\bigoplus_{U,V \in \Irr(\cC)}} \cN(U,V)_* \Biggr]^*.
\]
By Proposition~\ref{minimalidemp}, $S$ has a minimal idempotent $\Psi$.

Let us make sense of the image of $\Psi$ as a C$^*$-$\cC$-module category.
For each $U$ and $V$ in $\cC$, set
\[
  \Psi(\cN)(U, V) = \Psi(\cN(U, V)).
\]
When $U = V$, we get a structure of C$^*$-algebra on this space by the Choi--Effros product $S \cdot T = \Psi(S T)$.
Generally, by considering $X = U \oplus V \oplus W$ and $\Psi(\cN)(X, X)$, we get the composition maps
\[
  \Psi(\cN)(V, W) \times \Psi(\cN)(U, V) \to \Psi(U, W)
\]
that defines a C$^*$-category $\Psi(\cN)$.
Moreover, using the fact that $\Psi$ is a multiplier, we obtain maps
\[
  \Psi(\cN)(V, W) \times \cC(V', W') \to \Psi(\cN)(V \otimes V', W \otimes W')
\]
that defines a structure of $\cC$-module category on $\Psi(\cN)$.

We next claim that $\Psi(\cN)$ is an injective envelope for $\cM$.
Again by construction $i$ induces maps
\[
  \cM(m \stmryolt V, m \stmryolt W) \to \Psi(\cN)(V, W),
\]
and by the multiplicative domain argument this is a functor of C$^*$-categories.
It is straightforward to check the compatibility with $\cC$-module structures.
Moreover, the uniqueness of $\Id_{\Psi(\cN)}$ as a ucp multiplier stabilizing $\cM$ is obvious from the characterization of $\Psi$.
\end{proof}

Let $(\cM,m)$ be a cyclic $C^*$-$\cC$-module category, and let $\Psi\colon (\cM,m) \to (\cM,m)$ be a ucp multiplier.
The proof of Theorem~\ref{thm:injective-envelopes-mod-cat} shows that the image of $\Psi$ has a natural composition rule which makes it a cyclic $C^*$-$\cC$-module category: this result does not depend on minimality of $\Psi$, and it will be used later again to prove existence of injective envelopes for a class of bimodule categories, which can be equivalently understood as module categories over $\cC \boxtimes \cC^{\text{op}}$.

\section{Operator system theory for centrally pointed bimodule categories}\label{sec:op-sys-half-br-po-bim}

\subsection{Centrally pointed bimodules}

Let us review the ingredients of~\cite{YDpaper}.
Let $\cC$ be a rigid C$^*$-tensor category, and let $\cM$ be a $\cC$-bimodule C$^*$-category.
We assume, for simplicity and without loss of generality, that the bimodular structure is strict.
This data is equivalent to the data of a strict right $\cC^{mp} \boxtimes \cC$-module C$^*$-category, where $\cC^{mp}$ is the monoidal opposite of $\cC$.
An object $m \in \cM$ is said to be central if it is equipped with a family
\[
  \sigma_U\colon U \stmryogt m \simeq m \stmryolt U
\]
of unitary isomorphisms which is natural in $U \in \cC$, where $\stmryogt$ denotes the left $\cC$-action and $\stmryolt$ denotes the right $\cC$-action, and which  satisfies the half-braiding condition
\[
  \sigma_{U \otimes V} = (\sigma_U \stmryolt \id_V) \circ (\id_U \stmryogt \sigma_V) \ .
\]
Given a central object as above, we often work with the induced maps
\[
  \Sigma_{U;V,W}:\cM(m \stmryolt V,m \stmryolt W) \to \cM(m \stmryolt (U \otimes V), m \stmryolt (U \otimes W)) \ ,
\]
defined by
\begin{equation*}
  \Sigma_{U;V,W}(T) = (\sigma_U \stmryolt \id_W) (\id_U \stmryogt T) (\sigma_U^* \stmryolt \id_V).
\end{equation*}

\begin{definition}\label{def:cycliccentrallypointed bimodules}
A cyclic centrally pointed/centrally cyiclic bimodule C$^*$-category over $\cC$ is a C$^*$-$\cC$-bimodule category $\cM$ equipped with a central object $(m,\sigma)$ which is moreover cyclic for the right $\cC$-module structure.
\end{definition}

Observe that, in Definition~\ref{def:cycliccentrallypointed bimodules}, the object $m$ is cyclic for the right $\cC$-module structure if and only if it is cyclic for the left $\cC$-module structure, due to the central structure $\sigma$.

The structures and definitions above have a purely algebraic version, obtained by ignoring the $*$ and the Banach space structures.

\begin{definition}
Let $(\cM, m)$ and $(\cM', m')$ be centrally pointed bimodule categories.
A \emph{central functor} between them is a bimodule functor $F \colon \cM \to \cM'$ endowed with an isomorphism $F_0 \colon m' \to F(m)$ which is compatible with central generators and structure morphisms, in the sense that the diagram
\begin{equation*}
  \begin{tikzcd}
    U \stmryogt m' \arrow[r, "\sigma'_U"] \arrow[d] & m' \stmryolt U \arrow[d]\\ U \stmryogt F(m) \arrow[d,"{\tensor[_2]{F}{}}"'] & F(m) \stmryolt U \arrow[d,"F_2"]\\ F(U \stmryogt m) \arrow[r,"F(\sigma_U)"] & F(m \stmryolt U)\\
  \end{tikzcd}
\end{equation*}
is commutative.
In the above diagram, $\tensor[_2]{F}{}$ and $F_2$ denote left and right module structures of $F$, respectively.
If $\cM$ and $\cM'$ are C$^*$-$\cC$-bimodule categories, $F$ is $*$-preserving and $F_0$ is unitary, we say that $F$ is a functor of central C$^*$-$\cC$-bimodule categories or that it is a bimodule central bimodule $*$-functor.
\end{definition}

\begin{definition}\label{def:nat-trans-bimod-functors}
Let $F$ and $F'$ be central functors from $(\cM, m)$ to $(\cM, m')$.
A \emph{natural transformation of bimodule functors} $\alpha\colon F \to F'$ is given by a natural transformation of the underlying linear functors $\alpha_X \colon F(X) \to F'(X)$ for $X \in \cM$ such that the diagrams
\[
  \begin{tikzcd}
    U \stmryogt F(X) \arrow[r, "{\tensor[_2]{F}{}}"] \arrow[d, "\id_U \otimes \alpha_X"'] & F(U \stmryogt X) \arrow[d, "\alpha_{U \stmryogt X}"]\\ U \stmryogt F'(X) \arrow[r, "{\tensor[_2]{F}{^\prime}}"] & F'(U \stmryogt X)
  \end{tikzcd}
  \quad
  \begin{tikzcd}
    F(X) \stmryolt U \arrow[r, "F_2"] \arrow[d, "\alpha_X \otimes \id_U"'] & F(X \stmryolt U) \arrow[d, "\alpha_{X \stmryolt U}"]\\ F'(X) \stmryolt U \arrow[r, "F'_2"] & F'(X \stmryolt U)
  \end{tikzcd}
  \quad
  \begin{tikzcd}
    m' \arrow[r, "F_0"] \arrow[dr, "F'_0"'] & F(m) \arrow[d, "\alpha_m"]\\ & F'(m)
  \end{tikzcd}
\]
\end{definition}
commute.

\begin{definition}\label{def:centrallypointedcyclicbimodules}
Denote by $\CB(\cC)$ the category of centrally pointed C$^*$-$\cC$-bimodule categories and central bimodule $*$-functors, and by $\cCB(\cC)$ the subcategory of cyclic centrally pointed C$^*$-$\cC$-bimodule categories.
\end{definition}

The main result of~\cite{YDpaper} is the following correspondence.

\begin{theorem}[\cite{YDpaper}*{Theorem 4.2}]\label{thm:duality-YD-alg-centr-bimods}
The category of unital Yetter--Drinfeld $G$-C$^*$-algebras and equiavariant $*$-homomorphisms is equivalent to $\cCB(\Rep(G))$.
\end{theorem}

Recall from Section~\ref{sec:qg-prelim} that, for each unital $G$-C$^*$-algebra $A$, we have the category $\cD_A$ of $G$-equivariant and finitely generated projective right Hilbert $A$-modules, which admit natural right module category structure over $\Rep(G)$.
If $A$ is in addition an Yetter--Drinfeld algebra, then there is an extra left module structure on $\cD_A$, and $A$ as an object of $\cD_A$ becomes a central generator.
Concretely, the left action of $V \in \Rep(G)$ on $X \in \cD_A$ is given by
\[
  V \stmryogt X = X \otimes_A (H_V \otimes A),
\]
where the right $A$-module structure on $H_V \otimes A$ is defined by
\begin{equation}\label{eq:right-RepG-mod-str}
  \pi_V \colon A \to \End_A(H_V \otimes A), a \mapsto \sum_{i, j} m^V_{i j} \otimes (v_{i j} \rhd a),
\end{equation}
where we write the representation $V$ as $V = \sum_{i, j} m^V_{i j} \otimes v_{i j} \in \cB(H_V) \otimes \cO(G)$.
Conversely, let $(\cM, m)$ be a centrally pointed bimodule category over $\Rep(G)$, and $A$ be the corresponding Yetter--Drinfeld $G$-C$^*$-algebra.
Given $\xi$, $\zeta \in H_U$, $\eta \in H_V$, and $T \in \cM(m, m \stmryolt V)$, the element
\begin{equation}\label{eq:OG-action-from-centr-bimod}
  \overline{\xi \otimes \eta \otimes \overline{\rho_U^{-1/2} \zeta}} \otimes (\sigma_U \otimes \id) (\id \otimes T \otimes \id) (\sigma_U^{-1} \stmryogt \id) (\id_m \stmryolt \bar{R}_U)
\end{equation}
in $\bar{H}_{U \otimes V \otimes \bar{U}} \otimes \cM(m, m \stmryolt (U \otimes V \otimes \bar{U}))$ (we suppressed the associators) represents the action of $\bar{\xi} \otimes \zeta \in \bar{H}_U \otimes H_U \subset \cO(G)$ to $\bar{\eta} \otimes T \in \bar{H}_V \otimes \cM(m, m \stmryolt V) \subset A$.

\medskip
Let us continue with a few more preliminary materials.
An analogue of the canonical equivariant embedding $\C \to A$ for a unital Yetter--Drinfeld $G$-C$^*$-algebra $A$ is the following.

\begin{proposition}\label{prop:C-embeds-into-half-braided-mod-cat}
Let $(\cM, m)$ a centrally pointed bimodule category.
There is a embedding of centrally pointed bimodule categories $F \colon (\cC,1_\cC) \to (\cM,m)$ which sends $U$ to $m \stmryolt U$.
\end{proposition}

\begin{proof}
For simplicity we assume that the right action of $\cC$ on $\cM$ is strict.
Then we get a right $\cC$-module functor $F(U) = m \stmryolt U$ with $F_2 = \id_{m \stmryolt U}$.
We extend it to a bimodule functor by setting
\[
  {\tensor[_2]{F}{}} = \sigma_U \otimes \id_V\colon (U \stmryogt m) \stmryolt V \to m \stmryolt U \stmryolt V = m \stmryolt (U \otimes V).
\]
Then consistency conditions of ${\tensor[_2]{F}{}}$ follow from the braid relations for $\sigma$.
\end{proof}

\begin{definition}\label{def:C-hat-U}
Let $\cM$ be a right C$^*$-$\cC$-module category, and $m \in \cM$.
We denote by $\hat{\cM}_m$ the idempotent completion of $\cC$ with the enlarged morphism sets
\begin{equation}\label{eq:hat-M-morphism-set}
  \hat{\cM}_m(V,W) = \Nat_b ( m \stmryolt \iota \otimes V, m \stmryolt \iota \otimes W) \simeq \ell^\infty\mhyph\smashoperator[l]{\prod_{X \in \Irr(\cC)}} \cM(m \stmryolt (X \otimes V), m \stmryolt (X \otimes W)).
\end{equation}
\end{definition}

This is a bimodule category, with $1$ (which corresponds to $m \in \cM$) being a central generator.
The bimodule structure of $\hat\cM_m$ is as follows.
At the level of the objects it is induced by the tensor structure of $\cC$ At the level of morphisms, for $T \in \hat\cM_m(V,W)$, $X,Y,Z \in \cC$ and $\phi \in \cC(X,Y)$, we define the $Z$-component of the right action $T \stmryolt  \phi$ of $\phi$ on $T$ by
\[
  (T \stmryolt  \phi)_Z = T_Z \stmryolt  \phi \ .
\]
On the right hand side of the above equation we have used the module structure of $\cM$.
The left action $\phi \stmryogt T$, on the other hand, has the $Z$ component defined by the commutative diagram
\begin{equation*}
\begin{tikzcd}[row sep=small, column sep=small]
  m \stmryolt  Z \triangleleft(X \otimes V) \arrow[rrrr, "(\phi \stmryogt T)_Z"] \arrow[dd] &  & &  & m \stmryolt  Z \stmryolt  (Y \otimes W)  \\
   &  &  &  & \\
  m \stmryolt (Z \otimes X) \stmryolt  V \arrow[rr, "T_{Z \otimes X}"'] &  & m \stmryolt (Z \otimes X) \stmryolt W \arrow[rr] &  & m \stmryolt  Z \stmryolt (X \otimes W). \arrow[uu, "\id_{m \stmryolt  Z} \stmryolt  (\phi \otimes \id_W)"']
\end{tikzcd}
\end{equation*}
Using the naturality of $T$, it is easy to check that these indeed define a $\cC$-bimodule structure on $\hat{\cM}_m$.
It is moreover compatible with the $*$-structure, so that $\hat{\cM}_m$ has a canonical structure of a C$^*$-$\cC$-bimodule category.
It is immediate that $1 \in \hat\cM_m$ is a central object, with $\sigma_U$ given by $\id_U$.

This is motivated by the `dual category' $\hat\cC$ introduced in~\cite{NY1}, which corresponds to case of $\cM = \cC$ and $m = 1_\cC$.
In this case there is a natural C$^*$-tensor structure on $\hat\cC$ (with nonsimple unit) such that $U \mapsto U$ is a C$^*$-tensor functor from $\cC$ to $\hat\cC$.
If we take $m = U$ instead, because of the centrality, the resulting category can be identified with $\hat\cC_U$, with morphism sets
\[
  \hat\cC_U(V, W) = \hat\cC(U\otimes V,U \otimes W) = \ell^\infty\mhyph\smashoperator[l]{\prod_{i \in \Irr(\cC)}} \cC(U_i \otimes U \otimes V, U_i \otimes U \otimes W)
\]
for $V, W \in \cC$.
The evaluation at the unit $1_\cC$ gives a multiplier $\hat{\cC}_U \to \cC_U$ that is a conditional expectation (completely positive idempotent onto $\cC_U$, in the sense of~\cite{jp1}).
In terms of compact quantum group actions, we have the following.

\begin{proposition}
Suppose $\cC = \Rep(G)$ for some compact quantum group $G$.
For a finite dimensional unitary representation $(H, U)$ of $G$, the underlying right C$^*$-$\Rep(G)$-module category of $\hat{\cC}_U$ corresponds to the $G$-C$^*$-algebra $\cR(\cB(H) \vnotimes \ell^\infty(\hat{G}))$, the regular part of the $G$-W$^*$-algebra $\cB(H) \vnotimes \ell^{\infty}(\hat{G})$.
\end{proposition}

\subsection{Multipliers and injectivity}

\begin{definition}
Let $\cM_1$ and $\cM_2$ be centrally pointed $\cC$-bimodule categories, with central generators $m_1$ and $m_2$.
A \emph{central ucp $\cC$-linear multiplier} $F\colon \cM_1 \to \cM_2$ is a ucp right $\cC$-linear multiplier $F = \{F_{U,V}\colon \cM_1( m_1 \stmryolt U, m_1 \stmryolt V) \to \cM_2(m_2 \stmryolt U, m_2 \stmryolt V) \}_{U,V}$ such that
\begin{equation}\label{eq:ucp-mult-for-centr-br-bimod}
  F_{UV,UW} \Sigma_{U;V,W} = \Sigma_{U;V,W}' F_{V,W} \quad (U,V,W \in \cC).
\end{equation}
We denote by $\CBOS(\cC)$ the category of centrally pointed $\cC$-bimodule categories and central ucp $\cC$-linear multipliers.
\end{definition}

\begin{remark}
In~\cite{YDpaper} we showed that the category of centrally pointed bimodule categories is equivalent, through trivialization of the central structures, to a category of pointed cyclic bimodules $(\cM, m)$ on which the left bimodule structure is an extension of $U \stmryogt (m \stmryolt V) = m \stmryolt U \stmryolt V$.
At the level of morphisms, the left action on such a trivialized category was given by
\[
  \id_U \stmryogt T = \Sigma_{U;V,W}(T) \quad (T \in \cM(m \stmryolt V, m \stmryolt W)).
\]
Equation~\eqref{eq:ucp-mult-for-centr-br-bimod} is saying, therefore, that $F$ is a ucp-multiplier of the corresponding bimodule categories with trivialized central structures.
\end{remark}

\begin{proposition}\label{prop:corresp-cent-br-ucp-mult-and-usual-mult}
Given $\cM \in \CBOS(\cC)$ and $\cM' \in \cMod_\cC(*)$, there is a bijective correspondence between the morphisms $ \cM \to \hat\cM'$ in $\CBOS(\cC)$ and the ucp-multipliers $\cM \to \cM'$ of pointed right $\cC$-modules.
\end{proposition}

\begin{proof}
Let $F$ be a ucp-multiplier $\cM \to \cM'$.
Given $T \in \cM(m \stmryolt V, m \stmryolt W)$ and $U \in \cC$, we define $\hat F_{V, W}(T)_U \in \cM'(m' \stmryolt U \stmryolt V, m' \stmryolt U \stmryolt W)$ by
\[
  \hat F_{V, W}(T)_U = F_{UV, UW}((\sigma_U \stmryolt \id_W) (\id_U \stmryogt T) (\sigma_U^{-1} \stmryolt \id_V)) = F_{UV,UW}( \Sigma_{U;V,W}(T)).
\]
This is natural in $U$ by construction, hence we get a multiplier $\hat F \colon \cM \to \hat\cM'$.
This is completely positive because $\sigma_U$ is unitary.

Moreover, $\Sigma'_{U; V, W}$ on $\hat\cM'$ is the left module structure map $S \mapsto \id_U \stmryogt S$.
Thus,~\eqref{eq:ucp-mult-for-centr-br-bimod} for $\hat F$ becomes
\[
  \hat F_{U V, U W}(\Sigma_{U; V, W}(T)) = \id_U \stmryogt \hat F_{V, W}(T) \quad (T \in \cM(m \stmryolt V, m \stmryolt W)).
\]
This follows from the multiplicativity of $\sigma_U$ in $U$.

In the other direction, given a ucp-multiplier $F \colon \cM \to \hat\cM'$ (satisfying~\eqref{eq:ucp-mult-for-centr-br-bimod}), we get a ucp-multiplier $\check F \colon \cM \to \cM'$ by $\check F_{V,W}(T) = F_{V, W}(T)_1$ for $T \in \cM(m \stmryolt  V, m \stmryolt W)$.

These correspondences are inverse to each other: since $\sigma_1 = \id_m$, if $F\colon \cM \to \cM'$ is a right-module ucp-multiplier, $G = \hat{F}$ satisfies
\[
  \check G_{V,W} = [\hat{F}_{V,W}]_1 = F_{V,W}.
\]
Conversely, suppose $F\colon \cM \to \hat{\cM}'$ is a central ucp-multiplier, and write $G = \check F$.
Then we have
\begin{multline*} [\hat G_{V,W}(T)]_U = \check{F}_{UV,UW}( \Sigma_{UV,UW}(T)) = F_{UV,UW}(\Sigma_{U;V,W}(T))_1 = \Sigma'_{U;V,W}(F_{V,W}(T))_1 \\ = [\id_U \stmryogt F_{V,W}(T)]_1 = [F_{V,W}(T)]_U
\end{multline*}
for $T \in \cM(m\stmryolt V, m \stmryolt W)$.
\end{proof}

Now, consider the notions of \emph{injectivity} and \emph{injective envelopes} for centrally pointed bimodule categories analogously to Definitions~\ref{def:mod-cat-inj} and~\ref{def:inj-envelope-mod-cat}, by restricting to central ucp-multipliers.
In other words, we consider only central ucp-extensions of central ucp-multipliers.

\begin{theorem}\label{centrallybraidedtheo2}
The categories $\hat{\cM}_{\cB(H)}$ are injective in $\CBOS(\cC)$ for $H \in \Hilb(\cC)$.
\end{theorem}

\begin{proof}
By Proposition~\ref{prop:corresp-cent-br-ucp-mult-and-usual-mult}, the central ucp-multipliers $\cM' \to \hat\cM_{\cB(H)}$ bijectively correspond to the right module ucp-multipliers $\cM' \to \cM_{\cB(H)}$.
We then get the claim by Theorem~\ref{thm:arveson-type-thm-for-mod-cat}.
\end{proof}

As a consequence, any object in $\CBOS(\cC)$ embeds into a injective object.

\begin{remark}
As the above proof shows, $\hat\cM$ is injective as a centrally pointed bimodule category whenever $\cM$ is an injective pointed right $\cC$-module category.
\end{remark}

\begin{remark}
Consider a right $\cC$-module $C^*$-category $\cM$ and an embedding of $\cM$ into $\cM_{\cB(H)}$, for some $H \in \Hilb(\cC)$.
The composition
\[
  \cM \to \cM_{\cB(H)} \to \hat\cM_{\cB(H)}
\]
is an analogue of the \emph{Poisson transform} considered in~\citelist{\cite{kksv1}\cite{HHN1}}, which we call a {\em categorical Poisson transform}.
\end{remark}

\begin{definition}
A central bimodule functor $F \colon \cM \to \cN$ is said to be \emph{rigid} when $\Id_\cN$ is the only central ucp multiplier $G \colon \cN \to \cN$ satisfying $G F = F$.
\end{definition}

\begin{theorem}\label{thm:injective-envelope-cb-mod-cat}
Every object in $\CBOS(\cC)$ has an injective envelope $\cI_{\CBOS}(\cM)$.
Moreover the embedding $i \colon \cM \to \cI_{\CBOS}(\cM)$ is rigid.
\end{theorem}

\begin{proof}
The first claim can be proved with a slight modification of the proof of Theorem~\ref{thm:injective-envelopes-mod-cat}.
Start with a centrally pointed bimodule category $(\cM,m)$, and embed it into $\hat{\cM}_{\cB(H)}$ for some $H \in \Hilb(\cC)$.

Then consider the semigroup $S$ of central ucp multipliers from $\hat{\cM}_{\cB(H)}$ to itself that are identity on the image of $\cM$.
This is still a closed convex set for the weak$^*$ topology with respect to the presentation~\eqref{eq:hat-M-morphism-set}, ensuring then the existence of a minimal idempotent $\Phi \in S$.
Consider the image of $\Phi$, endowed with the Choi--Effros product.
Since $\Phi$ is central, its image inherits the centrally pointed bimodule structure of $\hat{\cM}_{\cB(H)}$.
This is a model of $\cI = \cI_{\CBOS}(\cM)$.

The rest of the proof is also a close analogue of the usual one for injective envelope of C$^*$-algebras.
Let $\Psi\colon \cI \to \cI$ be a central ucp multiplier with $\Psi i = i$, and put $\Phi' = \Psi \Phi$.
Then by the complete contractivity of $\Psi$, we have
\[
  \norm{\Phi'(T)} \le \norm{\Phi(T)}
\]
for any $U, V \in \cC$ and any $T \in \hat{\cM}_{\cB(H)}(U, V)$.
By the minimality of $\Phi$, we obtain $\Phi' = \Phi$, which means $\Psi = \Id_\cI$.
\end{proof}

Now the parallel between Section~\ref{sec:inj-env-YD-G-algs} up to the correspondence from Theorem~\ref{thm:duality-YD-alg-centr-bimods} should be clear.
The role of $\cR(\cB(H) \otimes \ell^\infty (\hat{G}))$ is played by $\hat{\cM}_{\cB(H)}$, and in the case of $\cC = \Rep(G)$ and $\cM = \cD_A$ for a Yetter--Drinfeld $G$-C$^*$-algebra $A$ with an equivariant representation $A \to \cB(H)$, the constructions really agree.

\section{Boundary theory for centrally pointed bimodule categories}\label{sec:boundarytheory}

Let us quickly recall the relevant concepts from~\cite{HHN1}: A \emph{$\cC$-tensor category} is a C$^*$-tensor category $\cD$ endowed with a dominant faithful C$^*$-tensor functor $F \colon \cC \to \cD$.
(For simplicity we assume that $\cC$ is a subcategory of $\cD$, and that any object of $\cD$ is a subobject of some $U \in \cC$.)
A \emph{$\cC$-linear transformation} $\Theta \colon \cD_1 \to \cD_2$ between $\cC$-tensor categories is given by a family of linear maps
\[
  \Theta_{U, V}\colon \cD_1(U, V) \to \cD_2(U, V) \quad (U, V \in \cC)
\]
satisfying
\begin{gather*}
  \Theta_{U_2, V_1}(S_1 T S_2) = S_1 \Theta_{V_2, U_1}(T) S_2 \quad (T \in \cD_1(V_2, U_1), S_i \in \cC(U_i, V_i)),\\
  \Theta_{Z \otimes U \otimes Y, Z \otimes V \otimes Y}(\id_Z \otimes T \otimes \id_Y) = \id_Z \otimes \Theta_{U,V}(T) \otimes \id_Y \quad (T \in \cD_1(U, V)).
\end{gather*}
A $\cC$-linear transformation $\Theta$ is \emph{ucp} if $\Theta_{U, U}$ is ucp for all $U \in \cC$.
The \emph{$\cC$-injectivity} for $\cC$-tensor categories is defined using this class of maps.
Finally, the \emph{Furstenberg--Hamana boundary} $\partial_\FH(\cC)$ of $\cC$ is an injective $\cC$-tensor category such that any $\cC$-linear transformation $\partial_\FH(\cC) \to \cD$ is completely isometric.
Concretely, $\partial_\FH(\cC)$ can be realized inside $\hat\cC$ using Proposition~\ref{minimalidemp} applied to the semigroup of $\cC$-linear ucp transformations from $\hat\cC$ to itself.

Now, observe that any $\cC$-tensor category $\cD$ has a canonical structure of centrally pointed $\cC$-bimodule category: as the generating object we take the tensor unit, and the central structure $\sigma_U \colon U \otimes 1_\cC \to 1_\cC \to U$ is given by the composition of structure morphisms of the monoidal unit.
Under this correspondence, a ucp $\cC$-linear transformations between $\cC$-tensor categories is exactly a ucp central $\cC$-linear transformation.
This motivates the following definition of boundary bimodules.

\begin{definition}\label{def:boundarycategories}
A centrally pointed $\cC$-bimodule category $\cM$ is called a \emph{boundary} category when any central ucp multiplier $\cM \to \cN$ to any centrally pointed bimodule category $\cN$ is completely isometric.
\end{definition}

\begin{theorem}\label{thm:charac-bnd}
Every boundary $\cM$ is a centrally pointed subcategory of $\partial_\FH(\cC)$.
\end{theorem}

\begin{proof}
By Proposition~\ref{prop:C-embeds-into-half-braided-mod-cat}, we have central embedding of $\cC$ into $\cM$.
Since $\cM$ is a boundary and $\partial_{FH}(\cC)$ is injective, we get a completely isometric central ucp multiplier $i\colon \cM \to \partial_{FH}(\cC)$.

Take an embedding $\pi$ of $\cM$ to $\hat{\cM}_{\cB(H)}$ for some $H \in \Hilb(\cC)$, as a centrally pointed bimodule category.
By the injectivity of $\hat{\cM}_{\cB(H)}$, the multiplier $i$ extends to a ucp central multiplier $\Phi\colon \partial_\FH(\cC) \to \hat{\cM}_{\cB(H)}$.
Conversely, by the injectivity of $\partial_{FH}(\cC)$, we also get a ucp central map $\Psi\colon \hat{\cM}_{\cB(H)} \to \partial_\FH(\cC)$.

The composition $\Psi \Phi$ is a $\cC$-linear ucp transformation on $\partial_\FH(\cC)$.
Thus, $\Phi$ must be completely isometric.
Moreover, by the above construction of $\partial_{FH}\cC$ via Proposition~\ref{minimalidemp}, $\Psi \Phi$ must be given by the identity functor.
This implies that $E = \Phi \Psi$ is a conditional expectation on $\hat{\cM}_{\cB(H)}$, with image $\Phi(\partial_\FH(\cC))$.
As before, the Choi--Effros product
\[
  \Phi(g) \cdot \Phi(f) = E(\Phi(g) \Phi(f))
\]
turns this image into a category.
Since we worked with central ucp central multipliers throughout, this category is again centrally pointed.

By construction $\pi(\cM)$ is contained in $\Phi(\partial_{FH}(\cC))$.
Moreover, the Choi--Effros product coincides with the original product of $\pi(\cM)$, by
\[
  \pi(g) \cdot \pi(f) = E(\pi(g) \pi(f)) = \Phi( \Psi(\pi(g f))) = \Phi(i(g f)) = \pi(g f) = \pi(g) \pi(f).
\]
Thus $\pi$ defines a functor $\cM \to \Phi(\partial_\FH(\cC))$, which can be upgraded to a central functor of centrally pointed bimodules.
By the boundary property of $\cM$, this must be an embedding.
\end{proof}

\begin{corollary}
The boundary objects in $\CBOS(\cC)$ are exactly the centrally pointed $\cC$-bimodule subcategories of $\partial_\FH(\cC)$.
\end{corollary}

\subsection{Monoidal invariance of \texorpdfstring{$D(G)$}{D(G)}-boundary actions}

\begin{proposition}[cf.~\cite{HHN1}, Proposition 4.5]
Let $G$ be a reduced compact quantum group.
Let $A_1$ and $A_2$ be continuous Yetter--Drinfeld $G$-C$^*$-algebras.
Denote be $(\cM_i,m_i)$ the cyclic centrally pointed bimodule C$^*$-category corresponding to $A_i$.
A ucp $G$-equivariant map $\phi\colon A_1 \to A_2$ is $\hat{G}$-equivariant if and only if the corresponding right ucp multiplier $\Phi\colon (\cM_1,m_1) \to (\cM_2,m_2)$ is a central ucp multiplier.
\end{proposition}

\begin{proof}
Suppose $\phi$ is $\hat{G}$-equivariant.
Since the right $\Rep(G)$-module structure on the categories $\cM_i = \cD_{A_i}$ are given naturally in terms of the $\cO(G)$-module structure on $A_i$ as in~\eqref{eq:right-RepG-mod-str}, we have the centrality of $\Phi$.
Conversely, if $\Phi$ is central, equation~\eqref{eq:OG-action-from-centr-bimod} implies that $\phi$ should be a $\cO(G)$-homomorphism.
\end{proof}

\begin{theorem}\label{thm:invarianceofboundarytheory}
If $G$ and $G'$ are monoidally equivalent reduced compact quantum groups, then the categories of boundary actions of the Drinfeld double $D(G)$ of $G$ is equivalent to the category of boundary actions of $D(G')$, the Drinfeld double of $G'$.
\end{theorem}

\begin{proof}
Given continuous Yetter--Drinfeld $G$-C$^*$-algebras $A_1$ and $A_2$ and a $D(G)$-equivariant ucp map $\phi\colon A_1 \to A_2$, the corresponding central ucp-multiplier $\Phi\colon (\cM_1,m_1) \to (\cM_2,m_2)$ is completely isometric if and only if $\phi$ is completely isometric.
Indeed, observing that for this statement it is enough to regard $\phi$ as a $G$-equivariant ucp map and $\Phi$ as a right ucp multiplier, this statement is already proven in~\cite{HHN1}*{Proposition 4.3}.

By the above proposition, a continuous Yetter--Drinfeld C$^*$-algebra $A$ is a $D(G)$-boundary action if and only if the corresponding cyclic centrally pointed bimodule $(\cM,m)$ is a boundary in the sense of Definition~\ref{def:boundarycategories}.
Therefore, the category of $D(G)$-boundary actions is equivalent to the category of $\Rep(G)$-boundary categories.
The latter is manifestly an invariant of $\Rep(G)$.
\end{proof}

\appendix

\section{Intrinsic characterization of \texorpdfstring{C$^*$}{C*}-algebra objects}\label{app:intr-char-C-star-alg}

Definition~\ref{def:c-star-alg-obj} does not give an intrinsic condition on C$^*$-algebra objects.
Here we give an intrinsic characterization of such structures, without directly referring to the corresponding module category.

When $A$ is an algebra object in $\Vect(\cC)$, let us denote the product of $a \in A(U)$ and $b \in A(V)$ by $a \cdot b \in A(V \otimes U)$.
We also write $A_0 = A(1_\cC)$.

Now, let $A$ be a \emph{$*$-algebra object} in $\Vect(\cC)$~\cite{jp1}.
This means that there is a family of conjugate linear maps
\[
  A(U) \to A(\bar U), \quad a \mapsto a^\natural
\]
(denoted by $j_U$ in~\cite{jp1}) satisfying the following conditions:
\begin{itemize}
\item the naturality
  \begin{equation}\label{eq:dagger-alg-obj-naturality}
    A(T)(a)^\natural = A(T^{* \vee})(a^\natural) \quad (T \in \cC(U, V), a \in A(V));
  \end{equation}
\item the involutivity $(a^\natural)^\natural = a$ up to the canonical isomorphism $\bar{\bar U} \simeq U$;
\item the unitality\footnote{There seems to be a typo in~\cite{jp1}.} $1_{A_0}^\natural = 1_{A_0}$ in $A_0$ up to the canonical choice $\bar 1_\cC = 1_\cC$; and
\item the antimultiplicativity $a^\natural \cdot b^\natural = (b \cdot a)^\natural$ up to the natural isomorphism $\bar U \otimes \bar V \simeq \overline{V \otimes U}$.
\end{itemize}
Here, we recall that $T^{* \vee} \in \cC(\bar U, \bar V)$ is the morphism characterized by
\[
  (\id_U \otimes T^{* \vee}) \bar R_U = (T^* \otimes \id_{\bar V}) \bar R_V.
\]

Such a structure induces a dagger category structure on the associated module category $\cM_A$.
In particular, $A(U \otimes \bar{U}) \simeq \cM_A(U_A, U_A)$ is a $*$-algebra for every $U \in \cC$.
First note that we have $*$-algebra embeddings
\[
  j_{U,V}\colon A(U \otimes \bar U) \to A(U \otimes V \otimes \bar V \otimes \bar U), \quad a \mapsto A(\id_U \otimes \bar R_V^* \otimes \id_{\bar U})(a),
\]
corresponding to
\[
  \cM_A(U_A, U_A) \to \cM_A(U_A \stmryolt V, U_A \stmryolt V), \quad T \mapsto T \stmryolt \id_V.
\]
This has a left inverse $a \mapsto d_V^{-1}
  A(\id_U \otimes \bar R_V \otimes \id_{\bar U})(a)$.
Moreover,
\begin{equation}\label{eq:inn-prod-for-star-alg-obj}
  \linnprod{A_0}{a}{b} = A(R_U)(a \cdot b^\natural)
\end{equation}
defines an $A_0$-valued inner product on the left $A_0$-module $A(U)$.

\begin{definition}
A \emph{pre-C$^*$-algebra object} in $\Vect(\cC)$ is a $*$-algebra object $A$ such that $A_0$ is a C$^*$-algebra and the inner product~\eqref{eq:inn-prod-for-star-alg-obj} is Hermitian, i.e., $\linnprod{A_0}{a}{a}$ is a positive element of $A_0$ for any $a \in A(U)$.
\end{definition}

Let us now assume that $A$ is a pre-C$^*$-algebra object.
By taking fiberwise completion, we may assume that $A(U)$ is a right Hilbert module over $A_0$.
Our goal is to see that this is a C$^*$-algebra object.

\begin{proposition}\label{prop:continuity-pre-C-star-alg-obj}
Let $A$ be a pre-C$^*$-algebra object in $\Vect(\cC)$.
Then we have
\[
  \norm{A(T)(a)} \le \norm{T} \norm{a} \quad (T \in \cC(V, U), a \in A(U))
\]
for the norm of right Hilbert $A_0$-modules.
Similarly, we have
\[
  \norm{a \cdot b} \le \norm{a} \norm{b} \quad (a \in A(U), b \in A(V)).
\]
\end{proposition}

\begin{proof}
We have
\[
  \linnprod{A_0}{A(T)(a)}{A(T)(a)} = \linnprod{A_0}{A(T T^*)(a)}{a}
\]
by~\eqref{eq:dagger-alg-obj-naturality} and the way we define inner product,~\eqref{eq:inn-prod-for-star-alg-obj}.
From $T T^* \le \norm{T}^2$, there is $S$ such that $\norm{T}^2 - T T^* = S S^*$.
Then we get
\[
  \norm{T}^2 \linnprod{A_0}{a}{a} - \linnprod{A_0}{A(T)(a)}{A(T)(a)} = \linnprod{A_0}{A(S)(a)}{A(S)(a)} \ge 0,
\]
hence the first claim.

As for the second claim, we can start from $\linnprod{A_0}{a}{a} \le \norm{a}^2$ and use a similar argument.
\end{proof}

\begin{theorem}
Let $A$ be a pre-C$^*$-algebra object in $\Vect(\cC)$.
Then the fiberwise completion of $A$ with respect to the Hilbert module structure over $A_0$ is a C$^*$-algebra object.
\end{theorem}

\begin{proof}
The natural action of $A(U \otimes \bar U)$ on $A(U)$ is by bounded and adjointable $A_0$-homomorphisms by Proposition~\ref{prop:continuity-pre-C-star-alg-obj}.
We can thus assume that each $A(U)$ is already a Hilbert $A_0$-module.

From the associativity of product $\cdot$ and the antimultiplicativity of involution $\natural$, we see that the natural action of $A(U \otimes \bar U \otimes U \otimes \bar U)$ on $A(U \otimes \bar U)$ is by adjointable homomorphisms.
Thus, the norm
\[
  \norm{a}' = \norm{j_{U,\bar U}(a)}_{\cB_{A_0}(A(U \otimes \bar U))}
\]
on $A(U \otimes \bar U)$ is a pre-C$^*$-norm.
We are going to show that this norm is equivalent to the Hilbert module norm on $A(U \otimes \bar U)$.
Then the fiberwise completion of $A$ gives a C$^*$-module category, hence the completion will be a C$^*$-algebra object in the sense of Definition~\ref{def:c-star-alg-obj}.

On one hand, the action of $j(a) = j_{U,\bar U}(a)$ on $\eta_U = A(\bar R_U^*)(1_{A_0})$ (the unit of $A(U \otimes \bar U)$) is $a$, hence we have
\[
  \norm{a}' \ge \norm{a}_{A(U \otimes \bar U)} \norm{\eta_U}_{A(U \otimes \bar U)}^{-1}.
\]
On the other, Proposition~\ref{prop:continuity-pre-C-star-alg-obj} implies
\[
  \norm{R_U} \norm{a}_{A(U \otimes \bar U)} \ge \norm{a}',
\]
and we obtain the claim.
\end{proof}

\raggedright


\begin{bibdiv}
\begin{biblist}

\bib{AV1}{misc}{
      author={Antoun, Jamie},
      author={Voigt, Christian},
       title={On bicolimits of {C$^*$}-categories},
         how={preprint},
        date={2020},
      eprint={\href{http://arxiv.org/abs/2006.06232}{\texttt{arXiv:2006.06232
  [math.OA]}}},
}

\bib{MR253059}{article}{
      author={Arveson, William~B.},
       title={Subalgebras of {$C^{\ast} $}-algebras},
        date={1969},
        ISSN={0001-5962},
     journal={Acta Math.},
      volume={123},
       pages={141\ndash 224},
         url={https://doi.org/10.1007/BF02392388},
         doi={10.1007/BF02392388},
      review={\MR{253059}},
}

\bib{biane}{article}{
      author={Biane, Philippe},
       title={Quantum random walk on the dual of {${\rm SU}(n)$}},
        date={1991},
        ISSN={0178-8051},
     journal={Probab. Theory Related Fields},
      volume={89},
      number={1},
       pages={117\ndash 129},
         url={http://dx.doi.org/10.1007/BF01225828},
         doi={10.1007/BF01225828},
      review={\MR{1109477}},
}

\bib{MR3735864}{article}{
      author={Breuillard, Emmanuel},
      author={Kalantar, Mehrdad},
      author={Kennedy, Matthew},
      author={Ozawa, Narutaka},
       title={{$C^*$}-simplicity and the unique trace property for discrete
  groups},
        date={2017},
        ISSN={0073-8301},
     journal={Publ. Math. Inst. Hautes \'{E}tudes Sci.},
      volume={126},
       pages={35\ndash 71},
         url={https://doi-org.ezproxy.uio.no/10.1007/s10240-017-0091-2},
         doi={10.1007/s10240-017-0091-2},
      review={\MR{3735864}},
}

\bib{MR0376726}{article}{
      author={Choi, Man~Duen},
       title={Completely positive linear maps on complex matrices},
        date={1975},
     journal={Linear Algebra and Appl.},
      volume={10},
       pages={285\ndash 290},
      review={\MR{0376726}},
}

\bib{ce1}{article}{
      author={Choi, Man-Duen},
      author={Effros, Edward~G.},
       title={Injectivity and operator spaces},
        date={1977},
     journal={J. Functional Analysis},
      volume={24},
      number={2},
       pages={156\ndash 209},
      review={\MR{0430809 (55 \#3814)}},
}

\bib{ydk1}{article}{
      author={De~Commer, Kenny},
      author={Yamashita, Makoto},
       title={Tannaka-{K}re\u\i n duality for compact quantum homogeneous
  spaces. {I}. {G}eneral theory},
        date={2013},
        ISSN={1201-561X},
     journal={Theory Appl. Categ.},
      volume={28},
       pages={No. 31, 1099\ndash 1138},
      eprint={\href{http://arxiv.org/abs/1211.6552}{\texttt{arXiv:1211.6552
  [math.OA]}}},
      review={\MR{3121622}},
}

\bib{furstenberg1}{article}{,
  title={A Poisson formula for semi-simple Lie groups},
  author={Furstenberg, Harry},
  journal={Annals of Mathematics},
  pages={335--386},
  year={1963},
  publisher={JSTOR}
}

\bib{jonesghosh1}{article}{
      author={Ghosh, Shamindra~Kumar},
      author={Jones, Corey},
       title={Annular representation theory for rigid {$C^*$}-tensor
  categories},
        date={2016},
        ISSN={0022-1236},
     journal={J. Funct. Anal.},
      volume={270},
      number={4},
       pages={1537\ndash 1584},
      eprint={\href{http://arxiv.org/abs/1502.06543}{\texttt{arXiv:1502.06543
  [math.OA]}}},
         url={http://dx.doi.org/10.1016/j.jfa.2015.08.017},
         doi={10.1016/j.jfa.2015.08.017},
      review={\MR{3447719}},
}

\bib{MR365537}{article}{
      author={Glasner, Shmuel},
       title={Compressibility properties in topological dynamics},
        date={1975},
        ISSN={0002-9327},
     journal={Amer. J. Math.},
      volume={97},
       pages={148\ndash 171},
         url={https://doi.org/10.2307/2373665},
         doi={10.2307/2373665},
      review={\MR{365537}},
}

\bib{HHN1}{article}{
   author={Habbestad, Erik},
   author={Hataishi, Lucas},
   author={Neshveyev, Sergey},
   title={Noncommutative Poisson boundaries and Furstenberg-Hamana
   boundaries of Drinfeld doubles},
   language={English, with English and French summaries},
   journal={J. Math. Pures Appl. (9)},
   volume={159},
   date={2022},
   pages={313--347},
   issn={0021-7824},
   review={\MR{4377998}},
   doi={10.1016/j.matpur.2021.12.006},
}

\bib{h4}{article}{
  title={Injective envelopes of Banach modules},
  author={Hamana, Masamichi},
  journal={Tohoku Mathematical Journal, Second Series},
  volume={30},
  number={3},
  pages={439--453},
  year={1978},
  publisher={Mathematical Institute, Tohoku University}
}

\bib{h2}{article}{
title={Injective envelopes of C*-algebras},
  author={Hamana, Masamichi},
  journal={Journal of the Mathematical Society of Japan},
  volume={31},
  number={1},
  pages={181--197},
  year={1979},
  publisher={The Mathematical Society of Japan}
}

\bib{h1}{article}{
      author={Hamana, Masamichi},
       title={Injective envelopes of operator systems},
        date={1979},
        ISSN={0034-5318},
     journal={Publ. Res. Inst. Math. Sci.},
      volume={15},
      number={3},
       pages={773\ndash 785},
         url={https://doi.org/10.2977/prims/1195187876},
         doi={10.2977/prims/1195187876},
      review={\MR{566081}},
}

\bib{h3}{article}{
title={Injective envelopes of $C^*$-dynamical systems},
  author={Hamana, Masamichi},
  journal={Tohoku Mathematical Journal, Second Series},
  volume={37},
  number={4},
  pages={463--487},
  year={1985},
  publisher={Mathematical Institute, Tohoku University}
}

\bib{MR2985658}{article}{
      author={Hamana, Masamichi},
       title={Injective envelopes of dynamical systems},
        date={2011},
        ISSN={1880-6015},
     journal={Toyama Math. J.},
      volume={34},
       pages={23\ndash 86},
      review={\MR{2985658}},
}

\bib{YDpaper}{misc}{
author={Hataishi, Lucas},
author={Yamashita, Makoto},
title={Categorical dualtiy for Yetter--Drinfeld C$^*$-algebras. Beyond the braided-commutative case},
year={2025}
}

\bib{izumi1}{article}{
  title={Non-commutative Poisson boundaries and compact quantum group actions},
  author={Izumi, Masaki},
  journal={Advances in Mathematics},
  volume={169},
  number={1},
  pages={1--57},
  year={2002},
  publisher={Elsevier}
}

\bib{izumi2}{incollection}{
      author={Izumi, Masaki},
       title={Non-commutative {P}oisson boundaries},
        date={2004},
   booktitle={Discrete geometric analysis},
      series={Contemp. Math.},
      volume={347},
   publisher={Amer. Math. Soc.},
     address={Providence, RI},
       pages={69\ndash 81},
         url={http://dx.doi.org/10.1090/conm/347/06267},
         doi={10.1090/conm/347/06267},
      review={\MR{2077031 (2005e:46126)}},
}

\bib{int}{article}{
      author={Izumi, Masaki},
      author={Neshveyev, Sergey},
      author={Tuset, Lars},
       title={Poisson boundary of the dual of {${\rm SU}_q(n)$}},
        date={2006},
        ISSN={0010-3616},
     journal={Comm. Math. Phys.},
      volume={262},
      number={2},
       pages={505\ndash 531},
  eprint={\href{http://arxiv.org/abs/math/0402074}{\texttt{arXiv:math/0402074
  [math.OA]}}},
         url={http://dx.doi.org/10.1007/s00220-005-1439-x},
         doi={10.1007/s00220-005-1439-x},
      review={\MR{MR2200270 (2007f:58012)}},
}

\bib{jp1}{article}{
  title={Operator algebras in rigid C*-tensor categories},
  author={Jones, Corey},
  author={Penneys, David},
  journal={Communications in Mathematical Physics},
  volume={355},
  number={3},
  pages={1121--1188},
  year={2017},
  publisher={Springer}
}


\bib{kksv1}{article}{
   author={Kalantar, Mehrdad},
   author={Kasprzak, Pawe\l },
   author={Skalski, Adam},
   author={Vergnioux, Roland},
   title={Noncommutative Furstenberg boundary},
   journal={Anal. PDE},
   volume={15},
   date={2022},
   number={3},
   pages={795--842},
   issn={2157-5045},
   review={\MR{4442841}},
   doi={10.2140/apde.2022.15.795},
}

\bib{kk}{article}{
  title={Boundaries of reduced C$^*$-algebras of discrete groups},
  author={Kalantar, Mehrdad},
  author={Kennedy, Matthew},
  journal={Journal f{\"u}r die reine und angewandte Mathematik (Crelles Journal)},
  volume={2017},
  number={727},
  pages={247--267},
  year={2017},
  publisher={De Gruyter}
}

\bib{kennedy}{article}{
      author={Kennedy, Matthew},
       title={An intrinsic characterization of {$C^*$}-simplicity},
        date={2020},
        ISSN={0012-9593,1873-2151},
     journal={Ann. Sci. \'{E}c. Norm. Sup\'{e}r. (4)},
      volume={53},
      number={5},
       pages={1105\ndash 1119},
         url={https://doi.org/10.24033/asens.2441},
         doi={10.24033/asens.2441},
      review={\MR{4174855}},
}

\bib{moore}{inproceedings}{
      author={Moore, Calvin~C.},
       title={Representations of solvable and nilpotent groups and harmonic
  analysis on nil and solvmanifolds},
        date={1973},
   booktitle={Harmonic analysis on homogeneous spaces ({P}roc. {S}ympos. {P}ure
  {M}ath., {V}ol. {XXVI}, {W}illiams {C}oll., {W}illiamstown, {M}ass., 1972)},
       pages={3\ndash 44},
      review={\MR{0385001}},
}

\bib{n1}{article}{
      author={Neshveyev, Sergey},
       title={Duality theory for nonergodic actions},
        date={2014},
        ISSN={1867-5778},
     journal={M{\"u}nster J. Math.},
      volume={7},
      number={2},
       pages={413\ndash 437},
      eprint={\href{http://arxiv.org/abs/1303.6207}{\texttt{arXiv:1303.6207
  [math.OA]}}},
      review={\MR{3426224}},
}

\bib{neshveyevtuset1}{article}{
      author={Neshveyev, Sergey},
      author={Tuset, Lars},
       title={The {M}artin boundary of a discrete quantum group},
        date={2004},
        ISSN={0075-4102},
     journal={J. Reine Angew. Math.},
      volume={568},
       pages={23\ndash 70},
  eprint={\href{http://arxiv.org/abs/math/0209270}{\texttt{arXiv:math/0209270
  [math.OA]}}},
         url={http://dx.doi.org/10.1515/crll.2004.018},
         doi={10.1515/crll.2004.018},
      review={\MR{MR2034922 (2006f:46058)}},
}

\bib{MR3204665}{book}{
      author={Neshveyev, Sergey},
      author={Tuset, Lars},
       title={Compact quantum groups and their representation categories},
      series={Cours Sp{\'e}cialis{\'e}s [Specialized Courses]},
   publisher={Soci{\'e}t{\'e} Math{\'e}matique de France, Paris},
        date={2013},
      volume={20},
        ISBN={978-2-85629-777-3},
      review={\MR{3204665}},
}

\bib{NY3}{article}{
      author={Neshveyev, Sergey},
      author={Yamashita, Makoto},
       title={Categorical duality for {Y}etter-{D}rinfeld algebras},
        date={2014},
        ISSN={1431-0635},
     journal={Doc. Math.},
      volume={19},
       pages={1105\ndash 1139},
      eprint={\href{http://arxiv.org/abs/1310.4407}{\texttt{arXiv:1310.4407
  [math.OA]}}},
      review={\MR{3291643}},
}

\bib{NY2}{article}{
  author={Neshveyev, Sergey},
   author={Yamashita, Makoto},
  title={Drinfeld center and representation theory for monoidal categories},
  journal={Communications in Mathematical Physics},
  volume={345},
  number={1},
  pages={385--434},
  year={2016},
  publisher={Springer}
}

\bib{NY1}{article}{
   author={Neshveyev, Sergey},
   author={Yamashita, Makoto},
   title={Poisson boundaries of monoidal categories},
   journal={Ann. Sci. \'{E}c. Norm. Sup\'{e}r. (4)},
   volume={50},
   date={2017},
   number={4},
   pages={927--972},
   issn={0012-9593},
   review={\MR{3679617}},
   doi={10.24033/asens.2335},
}

\bib{MR3933035}{article}{
      author={Neshveyev, Sergey},
      author={Yamashita, Makoto},
       title={Categorically {M}orita equivalent compact quantum groups},
        date={2018},
        ISSN={1431-0635},
     journal={Doc. Math.},
      volume={23},
       pages={2165\ndash 2216},
      review={\MR{3933035}},
}

\bib{ostrik03}{article}{
  title={Module categories, weak Hopf algebras and modular invariants},
  author={Ostrik, Victor},
  journal={Transformation groups},
  volume={8},
  number={2},
  pages={177--206},
  year={2003},
  publisher={Springer}
}

\bib{PV1}{article}{
 title={Representation Theory for Subfactors, $\lambda$-Lattices and C*-Tensor Categories},
  author={Popa, Sorin},
  author={Vaes, Stefaan},
  journal={Communications in Mathematical Physics},
  volume={340},
  number={3},
  pages={1239--1280},
  year={2015},
  publisher={Springer}
}

\bib{vaesvergnioux}{article}{
      author={Vaes, Stefaan},
      author={Vergnioux, Roland},
       title={The boundary of universal discrete quantum groups, exactness, and
  factoriality},
        date={2007},
        ISSN={0012-7094},
     journal={Duke Math. J.},
      volume={140},
      number={1},
       pages={35\ndash 84},
  eprint={\href{http://arxiv.org/abs/math/0509706}{\texttt{arXiv:math/0509706
  [math.OA]}}},
         url={http://dx.doi.org/10.1215/S0012-7094-07-14012-2},
         doi={10.1215/S0012-7094-07-14012-2},
      review={\MR{2355067 (2010a:46166)}},
}

\end{biblist}
\end{bibdiv}
\end{document}